\documentclass{scrartcl}
\title{Part II, Free Actions of Compact \\ Groups on C\Star Algebras}
\author{Kay Schwieger \and Stefan Wagner}	
\date{}

\KOMAoptions{parskip, paper=a4, abstract=on}
\usepackage[T1]{fontenc}
\usepackage{lmodern}
\usepackage{microtype}		
\usepackage[english]{babel}
\usepackage{amsmath, amsthm, amssymb, bbm, mathtools, nicefrac}
\usepackage{enumitem}
	\newlist{equivalence}{enumerate}{1}
	\setlist[equivalence]{label=\textnormal{(}\alph*\textnormal{)}}
\usepackage{xspace}
\usepackage{xifthen}
\usepackage{hyperref}

\theoremstyle{plain}
	\newtheorem{thm}{Theorem}[section]
	
	\newtheorem{lemma}[thm]{Lemma}
	\newtheorem{cor}[thm]{Corollary}

\theoremstyle{definition}
	\newtheorem{defn}[thm]{Definition}
	\newtheorem{rmk}[thm]{Remark}
	\newtheorem{expl}[thm]{Example}
\theoremstyle{remark}

\newcommand*{\N}{\mathbb N}		
\newcommand*{\Z}{\mathbb Z}		
\newcommand*{\C}{\mathbb C}		
\newcommand*{\one}{\mathbbm 1}		
\newcommand*{\tensor}{\otimes}		
	
\DeclarePairedDelimiter{\scal}{\langle}{\rangle}	
\DeclarePairedDelimiter{\norm}{\lVert}{\rVert} 	
\DeclareMathOperator{\Tr}{Tr}		
\DeclareMathOperator{\id}{id}		

\DeclareMathOperator{\SU}{SU}

\newcommand{\cf}{\mbox{cf.}\xspace}			
\newcommand{\eg}{\mbox{e.\,g.}\xspace}			
\newcommand*{\ie}{\mbox{i.\,e.}\xspace}			
\newcommand*{\ndash}{\nobreakdash-}

\newcommand*{\hilb}{\mathfrak}		
\newcommand*{\alg}{\mathcal}		
\newcommand*{\End}{\mathcal L}		
\newcommand*{\Star}{$^*$\ndash}
\newcommand*{\aA}{\alg A}		
\newcommand*{\aB}{\alg B}

\DeclareMathOperator{\Ad}{Ad}		

\newcommand*{\lscal}[2]{\scal*{#2}_{#1}}
\newcommand*{\Cont}{C}

\begin{document}

\author{
	Kay Schwieger \thanks{
		University of Helsinki, 
		\href{mailto:kay.schwieger@helsinki.fi}{\nolinkurl{kay.schwieger@helsinki.fi}}
	} \and 
	Stefan Wagner \thanks{
		Blekinge Tekniska H\"ogskola,
		\href{mailto:stefan.wagner@bth.se}{\nolinkurl{stefan.wagner@bth.se}}
	}
}
\maketitle

\begin{abstract}
	\noindent
	We study a simple class of free actions of non-Abelian groups on unital C\Star al\-ge\-bras, namely cleft actions. These are characterized by the fact that the associated noncommutative vector bundles are trivial. In particular, we provide a complete classification theory for these actions and describe its relations to classical principal bundles.

	\vspace*{0,5cm}
	

\end{abstract}

\section{Introduction}

In this presentation we investigate a special class of group actions on unital C\Star algebras. The experience with group actions on topological spaces shows that free actions are easier to understand and to classify. In this context a free action of a group $G$ on a space $P$ is typically regarded as a topological principal $G$-bundle over the quotient $P/G$. For instance, locally trivial bundles are, up to equivalence, characterized by the \v{C}ech cohomology $H^1(X,G)$. It is therefore reasonable to expect that free actions on noncommutative spaces are easier to understand and classify, too. 

In the first part of this series \cite{SchWa15} we investigated free actions of compact Abelian groups on unital C\Star algebras. This classification relies on the fact that the corresponding isotypic components are Morita self-equivalence over the fixed point algebra. For non-Abelian compact groups the bimodule structure is more subtle.  For this reason the current article concentrates on a simple class of free actions of non-Abelian groups, namely cleft actions. Regarded as noncommutative principal bundles, these actions are characterized by the fact that all associated noncommutative vector bundles are trivial. In particular, for such actions all Chern classes (\cf \cite{Bre04}) vanish. Although this property looks very limiting, in fact, many noncommutative phenomena already show up here. Therefore, cleft actions may be used as toy models of noncommutative principal bundles, \eg for Chern-Simmons actions (see \cite{Wi89,Pfa14}).

The earliest classification result for free actions of compact, but not necessarily Abelian groups goes back to Wassermann~\cite{Wass88}. He showed that free and ergodic actions, \ie, actions with full multiplicity and trivial fixed point space, are, up to equivalence, characterized by unitary 2-cocycles on the dual of the group. 
For finite groups Davydov~\cite{Da01} presented an alternative classification using classical group cohomology. 
Although free and ergodic actions only correspond to principal bundles over the singleton base space, these result yet show that noncommutative geometry admits more interesting examples than the classical theory.

The prototypes of free and ergodic actions are the quantum 2-tori $\aA_\theta$, $\theta \in \mathbb T$, equipped with the gauge action of $\mathbb T^2$. Varying $\theta \in \mathbb T$, these are in fact all non-equivalent free ergodic actions of the 2-torus $\mathbb T^2$. Non-ergodic examples can then be obtained by taking continuous bundles of ergodic actions. A~prominent example of this type is the gauge action of the 2-torus $\mathbb T^2$ on the Heisenberg group C\Star algebra $\aA$ (see \cite{EchNeOy09}). The algebra $\aA$ can be written as C\Star-bundle over~$\mathbb T$, where for each $\theta \in \mathbb T$ the fiber is the noncommutative 2-tori $\aA_\theta$.  The action of $\mathbb T^2$ on the algebra $\aA$ is the fiberwise gauge action. Echterhoff, Nest, and Oyono-Oyono \cite{EchNeOy09} proposed such bundles as noncommutative principal torus bundles. 
Concerning non-Abelian groups, the classification results of Wassermann can easily be extended to bundles of free ergodic actions. Such bundle actions are up to equivalence determined by the continuous family of unitary 2-cocycles corresponding to the ergodic actions in each fiber. 

Beyond the concept of bundles of ergodic actions there are multiple directions to introduce further noncommutativity. Any C\Star algebra obtained by forming a bundle of ergodic actions over a compact base space $X$ always contains $\Cont(X)$ in its center. In order to explore noncommutative principal bundles over a noncommutative base space, this requirement must be abolished. Without this restriction new examples are immediately available even for a classical base space. For instance, given a coaction $\gamma$ of $G$ on $X$, the crossed product $\Cont(X) \rtimes_\gamma \hat G$ with the dual action of $G$ is cleft (\cf Example  \ref{expl:twisted actions}) but $\Cont(X)$ is central only if the coaction is inner, \ie, if it is trivial up to a 1-cocycle (\cf \mbox{Example}~\ref{expl:SU_2}).

A second direction to explore is to not only consider actions of classical groups but of quantum groups. An algebraic approach to cleft actions of quantum groups (alias Hopf algebras) is already established in the theory of Hopf-Galois extensions (see \eg \cite{YuMi86, Schau04}). There cleft actions are free actions with a convolution invertible cleaving map. Doi \cite{Doi89} has shown that cleft actions can be written as a twisted crossed product and provided a classification for these crossed products. Also the work of Vaes and Vainerman \cite{VaVa03} should be mentioned, who studied cleft actions of locally compact Hopf von Neumann algebras. 
An essential part of our presentation will lift the algebraic constructions to the C\Star algebraic framework. In this article we restricts ourselves to classical compact groups for sake of a simple presentation and because most essential problems are already present in this restricted context. 

The article is structured as follows. After this introduction and some preliminaries, we recall the basic decomposition of a C\Star dynamical system into (generalized) isotypic components in Section~\ref{sec:decomp}. In Section~\ref{sec:cleft} we introduce cleft C\Star dynamical systems and a weaker notion, called \emph{weakly cleft}, and discuss their relations. Section~\ref{sec:factor_sys} presents the characterization of weakly cleft systems in terms of factor systems, alias cocycle actions. Moreover, we discuss some relations between the type of the dynamical system and the form of its factor system. In~particular, we classify all cleft topological (classical) principal bundles. Finally, in Section~\ref{sec:construction} we show that factor systems and weakly cleft C\Star dynamical systems are, up to equivalence, indeed in 1-1-correspondence by explicitly constructing the dynamical \mbox{system}.

\section{Preliminaries and Notations}
\label{sec:Prel}

Once and for the rest of the paper we fix a compact group $G$. All integrals over $G$ are taken with respect to the Haar probability measure. By a representation of $G$ we always mean a finite-dimensional unitary representation. For a representation $(\pi,V)$ we write $d_\pi$ for its dimension, for the dual representation we write $(\bar\pi, \bar V)$. 
The set of equivalence classes of irreducible representations will be denoted by $\hat G$. All our constructions behave naturally with respect to intertwiners and hence, for sake of brevity, we do not distinguish between a representation of $G$ and its equivalence class. 

Let $\aA$ be a unital C\Star algebra. For the unit of $\aA$ we write $\one_\aA$ or simply $\one$. For an element $u \in \aA$ we denote by $\Ad[u]:\aA \to \aA$ the map $x \mapsto u x u^*$.  All tensor products of C\Star algebras are taken with respect to the minimal tensor product. We will frequently deal with multiple tensor products of unital C\Star algebras $\aA$, $\aB$, and $\alg C$. If there is no ambiguity, we regard $\aA$, $\aB$, and $\alg C$ as subalgebras of $\aA \tensor \aB \tensor \alg C$ and extend maps on $\aA$, $\aB$, or $\alg C$ canonically by tensoring with the identity map. For sake of clarity we may occasionally use the leg numbering notation, \eg , for $x \in \aA \tensor \alg C$ we write $x_{13}$ to denote the corresponding element in \mbox{$\aA \tensor \aB \tensor \alg C$}. 

Our main focus in this paper will be on C\Star dynamical systems, by which we mean triples $(\aA, G, \alpha)$ consisting of a unital C\Star algebra $\aA$ together with a group of \Star automorphisms $\alpha_g:\aA \to \aA$, $g \in G$, such that for each $x \in \aA$ the map $g \mapsto \alpha_g(x)$ is continuous. We typically write $\aB = \aA^G$ for the fixed point algebra of the dynamical system and we denote by $P_0$ the associated the conditional expectation $P_0(x) := \int_G \alpha_g(x) \; dg$, $x \in \aA$. 
More general, for an irreducible representation $\pi \in \hat G$ we denote by $P_\pi:\aA \to \aA$ the $G$-equivariant projection onto the isotypic component $A(\pi):=P_\pi(\aA)$, which is given by
\begin{equation*}
	P_\pi(x) := d_\pi \int_G \Tr(\pi_g^*) \, \alpha_g(x) \; dg, 
	\qquad \qquad
	x \in \aA.
\end{equation*}
Two C\Star dynamical system $(\aA, G, \alpha)$ and $(\aA', G, \alpha')$ are called \emph{equivalent} if there is an isomorphism $\varphi:\aA \to \aA'$ with $\varphi \circ \alpha_g = \alpha'_g \circ \varphi$ for all $g \in G$. 

For the necessary background on modules of C\Star algebras we recommend \cite{Blackadar}, here we only briefly recall some relevant definition. For a C\Star algebra $\aB$ a right pre-Hilbert $\aB$-module is a right $\aB$-module $\hilb H$ equipped with a sesquilinear map $\scal{\cdot, \cdot}_\aB: \hilb H \times \hilb H \to \aB$ that satisfies the usual axioms of a definite inner product with $\aB$-linearity in the second component.%
\footnote{
	In the literature the notion is usually relaxed even further to pre-C\Star algebras and non-definite inner products. But we do not need this more general framework.
} 
We may define a norm on $\hilb H$ by putting $\norm{x}_{\hilb H} = \norm{\scal{x,x}_\aB}^{1/2}$. If $\hilb H$ is complete with respect to this norm then $\hilb H$ is called a right Hilbert $\aB$-module. A~linear operator \mbox{$T:\hilb H \to \hilb H$} on a right Hilbert $\aB$-module $\hilb H$ is called adjointable if there is an operator $T^\star:\hilb H \to \hilb H$ satisfying $\scal{Tx,y}_\aB = \scal{x, T^\star y}_\aB$ for all $x,y \in \hilb H$. Adjointable operators are automatically bounded but the converse does not hold. The set $\End(\hilb H)$ of all adjointable operators on a right Hilbert $\aB$-module is a C\Star algebra. 
A~correspondence over $\aB$, or a right Hilbert $\aB$-bimodule, is a $\aB$-bimodule $\hilb H$ equipped with a $\aB$-valued inner product $\scal{\cdot, \cdot}_\aB$ which turns it into a right Hilbert $\aB$-module such that the left action of $\aB$ on $\hilb H$ is via adjointable operators. For two correspondences $\hilb H$, $\hilb K$ over the same algebra $\aB$ we denote by $\hilb H \tensor_\aB \hilb K$ their tensor product, which is again a correspondence over $\aB$. The elementary tensors \mbox{$x \tensor y$} ($x \in \hilb H$, $y \in \hilb K$) are total in $\hilb H \tensor_\aB \hilb K$. The inner product on $\hilb H \tensor_\aB \hilb K$ is given by
\begin{equation*}
	\scal{x_1 \tensor y_1, \; x_2 \tensor y_2}_\aB = \scal{y_1, \; \scal{x_1,x_2}_\aB \,. \, y_2}_\aB
\end{equation*}
for all $x_1, x_2 \in \hilb  H_1$ and $y_1, y_2 \in \hilb K$.

\section{Decomposition of C$^*$-Dynamical Systems}
\label{sec:decomp}

As a background for later discussions we first would like to recall the general decomposition of C\Star dynamical systems $(\aA, G, \alpha)$. In analogy to the GNS-construction the conditional expectation $P_0$ onto the fixed point space $\aB := \aA^G$ allows to equip $\aA$ with the definite $\aB$-valued inner product
\begin{equation*}
	\scal{x,y}_\aB := P_0(x^* y) = \int_G \alpha_g(x^*y) \; dg
\end{equation*}
for $x,y \in \aA$. We write $L^2(\aA)$ for the right Hilbert $\aB$-module obtained by taking the completion of $\aA$ with respect to the corresponding norm. The C\Star algebra $\aA$ admits a faithful representation on $L^2(\aA)$ given by
\begin{equation*}
	\lambda: \aA \to \End \bigl( L^2(\aA) \bigr), 
	\quad
	\lambda(x)y := x \cdot y.
\end{equation*}
This allows to identify $\aA$ with the subalgebra $\lambda(\aA) \subseteq \End \bigl( L^2(\aA) \bigr)$ and we implicitly do so unless confusion arise. For each $g \in G$ we have a unitary operator on $L^2(\aA)$ given by $U_g x := \alpha_g(x)$ for $x \in \aA \subseteq L^2(\aA)$. The map $g \mapsto U_g$ is a strongly continuous representation of $G$ that implements the automorphisms $\alpha_g$, $g \in G$, in the sense that 
\begin{equation*}
	\lambda \bigl( \alpha_g(x) \bigr) = U_g \, \lambda(x) \, U_g^\star,
	\qquad \qquad x \in \aA.
\end{equation*}

As every representation of $G$, the algebra $\aA$ can be decomposed into its isotypic components with respect to the actions. More precisely, the sum of the isotypic components is direct and $\sum_{\pi \in \hat G}^{\text{alg}} A(\pi)$ is a dense \Star subalgebra of~$\aA$. Moreover, the isotypic components are mutually orthogonal, closed, linear subspaces of $L^2(\aA)$ with
\begin{equation*}
	L^2(\aA) = \overline{\bigoplus_{\pi \in \hat G}} A(\pi).
\end{equation*}
One aspect of this presentation will be the multiplication of $\aA$. It is worth noting that the multiplication is determined by the action of $A(\pi) \subseteq \aA$ on $A(\rho) \subseteq L^2(\aA)$ for all pairs of irreducible representations $\pi$ and $\rho$.

Instead of dealing with isotypic components, it will be more convenient to consider, for a representation $\pi$, the \emph{generalized isotypic component}
\begin{equation*}
	A_2(\pi) := \bigl\{ x \in \aA \tensor \End(V) \;\big|\; \pi_g \cdot \alpha_g(x) = x \quad \forall g\in G \bigr\}.
\end{equation*}
Obviously, $A_2(\pi)$ is a $\aB$-bimodule for the usual left and right multiplication. In addition, we may equip $A_2(\pi)$ with the $\aB$-valued inner product
\begin{equation*}
	\scal{x,y}_\aB := \tfrac{1}{d_\pi} (\id_\aA \tensor \Tr)(x^* y)
\end{equation*}
for $x,y \in A_2(\pi)$. Then the space $A_2(\pi)$ is a correspondence over $\aB$ (see \eg \cite{CoYa13} for completeness of the norm). If $\pi$ is irreducible, the map $x \mapsto (\id_\aA \tensor \Tr)(x)$ gives an isomorphism of correspondences from $A_2(\pi)$ to the dual isotypic component $A(\bar\pi)$ with inverse given by $y \mapsto d_\pi \int_G \alpha_g(y) \tensor \pi_g \, dg$. In the following we will frequently use this \mbox{identification}.

The multiplication between isotypic components is well captured by family of maps
\begin{gather*}
	m_{\pi, \rho}: A_2(\pi) \tensor_\aB A_2(\rho) \longrightarrow A_2(\pi \tensor \rho) \subseteq \aA \tensor \End(V) \tensor \End(W)
	\\
	m_{\pi, \rho}(x \tensor y) := x_{12} \cdot y_{13}
\end{gather*}
for pairs of representations $(\pi, V)$ and $(\rho, W)$ of~$G$. For an irreducible representation \mbox{$\sigma \in \hat G$} and a representation $\pi$ of $\hat G$ let us denote by \mbox{$P_{\sigma, \pi}: \End(V_\pi) \to \End(V_\sigma)$} the map $P_{\sigma, \pi}(x) := \sum_{k=1} v_k^* x v_k^{}$ with an orthonormal basis of intertwiners $v_1, \dots, v_m:V_\sigma \to V_\pi$. The map $P_{\sigma, \pi}$ does not depend on the choice of intertwiners. Then for irreducible representations $\pi, \rho \in \hat G$ and elements \mbox{$x \in A_2(\pi)$} and $y \in A_2(\rho)$ the operator $\lambda(x)$ takes the form
\begin{equation}
	\label{eq:right_mult}
	\lambda(x) y = \bigoplus_{\sigma \in \hat G} P_{\sigma, \pi \tensor \rho} \bigl( m_{\pi, \rho}(x \tensor y) \bigr).
\end{equation}
Since the sum of the isotypic components is norm dense in $\aA$, the C\Star algebra $\lambda(\aA)$ is generated by the set of these operators $\lambda(x)$ with \mbox{$x \in A_2(\pi)$}, $\pi \in \hat G$. In~particular, the multiplication of $\aA$ can be recovered from the maps $m_{\pi, \rho}$ for $\pi, \rho \in \hat G$ in this way.

\section{Cleft and Weakly Cleft C$^*$-Dynamical Systems}
\label{sec:cleft}

In principle, C\Star dynamical systems $(\aA, G, \alpha)$ with a given fixed point algebra $\aB$ can be classified in a functorial way in terms of the module structure of the generalized isotypic components
and their multiplicative relation (see \cite{Ne13}). In this presentation we will focus on the class of cleft actions. 

\begin{defn}
	A C\Star dynamical system $(\aA, G, \alpha)$ is called \emph{cleft} if for every irreducible representation $(\pi,V)$ of $G$ the set $A_2(\pi) \subseteq \aA \tensor \End(V)$ contains a unitary element.
\end{defn}


Cleft C\Star dynamical systems are precisely the so called semidual actions discussed in \cite{Wass89}. For such a C\Star dynamical system $(\aA, G, \alpha)$ it follows along the same lines as in \cite[Thm.\,10]{Wass89} that the crossed product $\aA \rtimes G$ is isomorphic to $\aA^G \tensor \mathbb K$, generalizing Green's Theorem (\cf \cite[Cor.\,15]{Green77} and \cite{EchNeOy09}). In the algebraic theory of actions of Hopf algebras these actions are up to equivalence given by twisted crossed products (\cf \cite{YuMi86,Doi89}, see also \cite{VaVa03}).  
Most arguments in our discussion rely on the following weaker hypothesis only and establishing the results in a slightly wider framework has some technical advantages later on. 

\begin{defn}
	A C\Star dynamical system $(\aA, G, \alpha)$ is called \emph{weakly cleft} if for every irreducible representation $(\pi,V)$ of $G$ the set $A_2(\pi) \subseteq \aA \tensor \End(V)$ contains an element $s$ such that
	\begin{align*}
		s^*s &= \one
		&
		&\text{and}
		&
		s s^* x &= x 
	\end{align*}
	for all $x \in A_2(\pi)$. We call such an element $s$ a \emph{non-degenerate isometry}.
\end{defn}
Unfortunately, we cannot present an examples of a weakly cleft but not cleft action. In fact, in simple examples our results show that weakly cleft dynamical systems are automatically cleft (see Lemma~\ref{lem:Abelian=>cleft} and~\ref{cor:comm=>free}). We do not yet know whether this holds in general. Most proofs of this article only rely on the weakly cleft assumption but can be simplified for cleft systems.
The advantage of dealing with weakly cleft actions is that this property can be characterized by the right Hilbert module structure of the (generalized) isotypic components. More precisely, the following lemma shows that the C\Star system $(\aA, G, \alpha)$ is weakly cleft if and only if each isotypic component $A(\pi)$, $\pi \in \hat G$, is a free right Hilbert $\aB$-module of rank $d_\pi^2$. In the ergodic case, $\aB = \C\one$, this is the same as saying that the Hilbert space $A(\pi)$ has its maximal dimension $d_\pi^2$, \ie, $A(\pi)$ has full \mbox{multiplicity}. 

\pagebreak[3]
\begin{lemma}
	\label{lem:cleft_isom}
	For an element $s \in A_2(\pi)$ the following statements are equivalent:
	\begin{equivalence}
	\item	\label{en:n.d. isometry}
		$s$ is a non-degenerated isometry.
	\item	\label{en:isomorphism}
		The map $\varphi:\aB \tensor \End(V_\pi) \to A_2(\pi)$, $\varphi(x) := s x$ is an isomorphism of right Hilbert $\aB$-modules.
	\end{equivalence}
\end{lemma}
\begin{proof}
	For one implication notice that for an isometry $s$ the map $\varphi(x) = sx$ is an isometry for the right inner product. If $s$ is non-degenerate, $\varphi$ admits the inverse $\varphi^{-1}(y) = s^*y$. Together this proves the implication from \ref{en:n.d. isometry} to \ref{en:isomorphism}. 
	For the converse implication, suppose that $\varphi$ is an isomorphism. Then the selfadjoint element $p := s^*s \in \aB \tensor \End(V)$ is a projection, since it satisfies
	\begin{equation*}
		\scal{p,x}_\aB
		= \scal{\varphi(p), \varphi(x)}_\aB
		= \tfrac{1}{d_\pi} (\id \tensor \Tr)(s^*ss^* sx)
		= \scal{\smash{p^2},x}_\aB
	\end{equation*}
	for all $x \in \aB \tensor \End(V)$. Injectivity of $\varphi$ then implies that $s$ is in fact an isometry because $\varphi(p) = ss^*s = s = \varphi(\one)$.
	It follows that the inverse of $\varphi$ is given by $\varphi^{-1}(x) = s^*x$ and hence we have $x = \varphi\bigl( \varphi^{-1}(x) \bigr) = ss^* x$ for all $x \in A_2(\pi)$. 
\end{proof}

To distinguish cleft and weakly cleft, let us introduce a third property, which is of great independent interest. Let $(\aA, G,\alpha)$ be a C\Star dynamical system. For a representation $(\pi, V)$ of $G$ we write $A_2(\pi) A_2(\pi)^*$ for the linear subspace generated by products $xy^*$ of elements $x,y\in A_2(\pi)$, and we put
\begin{equation*}
	C(\pi) := \{ x \in \aA \tensor \End(V) \;|\; \bigl( \alpha_g \tensor \Ad[\pi_g] \bigr)(x) = x \quad \forall g \in G\}.
\end{equation*}

\begin{defn}
	A~C\Star dynamical system $(\aA, G, \alpha)$ is called \emph{free} if for all representations $\pi$ of $G$ we have $A_2(\pi) A_2(\pi)^* = C(\pi)$.
\end{defn}

Free actions have attained special interest in the literature, see, \eg, \cite{Phi87,Rieffel91,Phi09,CoYa13}. For the definition given here and the relation to noncommutative principal bundles, we refer to \cite[Section 3]{SchWa15}. It should be noted that $A_2(\pi) A_2(\pi)^*$ is always a \Star ideal in $C(\pi)$. Therefore, the dynamical system is free if and only if the closure of $A_2(\pi)A_2(\pi)^*$ contains the unit $\one_{\aA \tensor \End(V)}$.

\pagebreak[3]
\begin{lemma}
	A C\Star dynamical system is cleft if and only if it is weakly cleft and free.
	\label{lem:cleft_n_sat}
\end{lemma}
\begin{proof}
	If $(\aA, G, \alpha)$ is cleft, it is obviously weakly cleft. Since for every representation $(\pi, V)$ of $G$ we find a unitary $u \in A_2(\pi)$, the set $A_2(\pi) A_2(\pi)^*$ contains $uu^* = \one_{\aA \tensor \End(V)}$ and hence the dynamical system is free.
	To show the converse, suppose that $(\aA, G, \alpha)$ is weakly cleft and free and let $(\pi, V)$ be a representation of $G$. Then we find a non-degenerated isometry $s \in A_2(\pi)$ and obtain
	\begin{equation*}
		A_2(\pi) A_2(\pi)^* 
		= \bigl( s \cdot \aB \tensor \End(V) \bigr) \cdot \bigl( s \tensor \aB \tensor \End(V) \bigr)^* 
		= s \cdot \aB \tensor \End(V) \cdot s^*.
	\end{equation*}
	This set is closed and, since the dynamical system is free, it contains the unit. It~follows that $s$ must have full range, \ie, $s$ is a unitary element and hence the system is cleft.
\end{proof}

\begin{lemma}
	For a compact Abelian group $G$, every weakly cleft C\Star dynamical system is cleft. 
	\label{lem:Abelian=>cleft}
\end{lemma}
\begin{proof}
	For an Abelian groups generalized isotypic components and the isotypic component of the dual representation literally coincide. If the C\Star dynamical system is weakly cleft, then for each $\pi \in \hat G$ we find an isometry $s \in A(\bar\pi)$. Hence for the dual representation $\bar\pi$ the set $A(\pi) A(\pi)^* = A(\bar\pi)^* A(\bar\pi)$ contain the element \mbox{$s^* s = \one$}.
\end{proof}

\begin{expl}
	\label{expl:trivial_bundle}
	We start with the most basic example as a prototype. Let $\aB$ be a unital C\Star algebra and denote by $C(G)$ the C\Star algebra of continuous functions on $G$. We consider the C\Star dynamical system $\bigl( \aB \tensor C(G), G, \id \tensor r \bigr)$ where the action on $C(G)$ is given by the right translation \mbox{$(r_g f)(h) := f(hg)$} for $f \in C(G)$ and \mbox{$h \in G$}. Clearly, the fixed point space of this system is $\aB = \aB \tensor 1_G$. We~may identify elements in a tensor product with $C(G)$ with functions on $G$ in the usual way. Then, for an irreducible representation $(\pi, V)$ of $G$, the correspondence $A_2(\pi)$ over $\aB$ is given by 
	\begin{equation*}
		A_2(\pi) = \{f:G \to \aB \tensor \End(V) \; \text{continuous} \;|\; \pi_g f(g) = f(e) \quad \forall g \in  G\}.
	\end{equation*}
	The function $u(g) := \pi_g^*$ then obviously is a unitary element in $A_2(\pi)$, which shows that the system is cleft.
\end{expl}

If $\aB$ is commutative, \ie, $\aB = C(X)$ for some compact Hausdorff space $X$, then the dynamical system of Example~\ref{expl:trivial_bundle} can be be understood as a trivial principal $G$\ndash bundle over the space~$X$. The next example shows that also non-trivial principal bundles may give rise to cleft actions. Moreover, in the later Corollary~\ref{cor:commutative_syst} we will provide a characterization of cleft topological principal bundles.

\begin{expl}
	\label{expl:non-trivial_bundle}
	Fix $n \in \N$ and denote by $C_n := \{\zeta \in \C \;|\; \zeta^n =1\}$ the group of $n$-th roots of unity. We consider the C\Star algebra $\aA := C(\mathbb T)$ of continuous function on the circle with the action of $C_n$ given by rotations, \ie, for $\zeta \in C_n$ and $f \in C(\mathbb T)$ put
	\begin{equation*}
		(\alpha_\zeta f)(z) := f(\zeta \cdot z),
		\qquad \qquad z \in \mathbb T.
	\end{equation*}
	For an irreducible representation of $C_n$, \ie, an element $k \in \Z / n\Z$, we have
	\begin{equation*}
		A_2(-k) = A(k) = \{f \in C(\mathbb T) \;|\; f(\zeta \cdot z) = \zeta^k \cdot f(z) \quad \forall z \in \mathbb T \}.
	\end{equation*}
	The action is cleft, because an invertible element in $A_2(-k)$ is, for instance, given by the function $f(z) := z^k$, $z \in \mathbb T$.
	This dynamical system can be understood as the non-trivial principal $C_n$-bundle corresponding to the $n$-fold covering \mbox{$p:\mathbb T \to \mathbb T$}, $p(z) := z^n$. 
	In fact, since all vector bundles over $\mathbb T$ are trivial (\cf \cite[Section~18]{Steen99}), every principal bundles over $\mathbb T$ with arbitrary compact structure group $G$ gives rise to a cleft C\Star dynamical system.
\end{expl}

\begin{expl}
	\label{expl:twisted actions}
	Consider a compact group $G$ and a closed normal subgroup $N$ and suppose that the action of $N$ on $\Cont(G)$ by right translations is cleft (\eg suppose that $G/N$ is finite). Furthermore, let $\delta:\aA \to M\bigl(\aA \tensor C^*(G))$ be a coaction of $G$ on a unital C\Star algebra~$\aA$. We recall that a \emph{twist} over $G/N$ is a unitary corepresentation $W\in M(\alg A\otimes C^*(G/N))$ of $G/N$ such that 
	$(\id \tensor q) \circ \delta = \Ad[W]$ and 
	$(\delta \tensor \id)(W) = W_{13}$,
	where $q$ denotes the natural \Star homomorphism $q:C^*(G) \to C^*(G/N)$. 
	Such a twist gives rise to an ideal $I_W$ of the crossed product $\aA \rtimes_\delta G$, called twisting ideal, which is invariant under the canonical action of $N$ on $\aA \rtimes_\delta G$ (see \cite{PhiRae94}). The \emph{twisted crossed product} $\aA \rtimes_{\delta,W} G := (\aA \rtimes_\delta G) / I_W$ then carries an action of $N$, which we denote by $\alpha$, that is, we obtain a C\Star dynamical system 
	\begin{equation*}
		(\aA \rtimes_{\delta,W} G, \; N, \;\alpha).
	\end{equation*}
	Since the natural embedding of $\Cont(G)$ into $\aA \rtimes_\delta G$ is $G$-equivariant, the factorized homomorphism $k_G:\Cont(G) \to \aA \rtimes_{\delta,W} G$ is $N$-equivariant. Finally, a few moments thought shows that being cleft is preserved under equivariant \Star homomorphisms and, therefore, it follows that the above C\Star dynamical system is cleft.
\end{expl}

\section{Factor Systems}
\label{sec:factor_sys}

Let $(\aA, G, \alpha)$ be a C\Star dynamical system with fixed point algebra $\aB$. Suppose for the moment that the system is weakly cleft, that is, for each irreducible representation $\pi \in \hat G$ we find a non-degenerate isometry $s_\pi \in A_2(\pi) \subseteq \aA \tensor \End(V_\pi)$. For the trivial representation, denoted by $1$, we pick $s_1 := \one_\aB$. We may extend this family of isometries to non-irreducible representations by decomposing each representation $(\pi, V)$ into a direct sum of irreducible representations $(\sigma_k, V_k)$ with intertwiners $v_k:V_k \to V$ for each \mbox{$1 \le k \le m$} and put \mbox{$s_\pi := \sum_{k=1}^m v_k^{} s_{\sigma_k} v_k^*$}.
It is easily checked that this provides a non-degenerate isometry in $A_2(\pi)$ and that the construction does not depend on the choice of intertwiners. 

By Lemma~\ref{lem:cleft_isom}, for each $\pi \in \hat G$ the space $A_2(\pi)$ is isomorphic to $\aB \tensor \End(V_\pi)$ as a right Hilbert $\aB$-module, but in general not as a left $\aB$-module. In order to describe the left action of $\aB$ consider the map
\begin{equation}
	\label{eq:gamma_concrete}
	\gamma_\pi: \aB \to \aB \tensor \End(V_\pi), 
	\quad
	\gamma_\pi(b) := s_\pi^* (b \tensor \one_\pi) s_\pi.
\end{equation}
The correspondence $A_2(\pi)$ is then isomorphic to the vector space $\aB \tensor \End(V_\pi)$ equipped with the usual right multiplication by $\aB$, the usual right $\aB$-valued inner product, and the left multiplication given by
\begin{equation*}
	b\,.\,x := \gamma_\pi(b) \, x,
\end{equation*}
for all $x \in \aB \tensor \End(V_\pi)$ and $b \in \aB$. 

\pagebreak[3]
\begin{lemma}
	The map $\gamma_\pi$ is a unital \Star homomorphism.
\end{lemma}
\begin{proof}
	Obviously, $\gamma_\pi(b^*) = \gamma_\pi(b)^*$ for all $b \in \aB$ and $\gamma_\pi$ is unital, since $s_\pi$ is an isometry. For each $b \in \aB$ the element $(b\tensor \one)s_\pi$ lies in $A_2(\pi)$. The non-degeneracy of $s_\pi$ then implies for all $b_1, b_2 \in \aB$ that
	\begin{equation*}
		\gamma_\pi(b_1) \gamma_\pi(b_2) 
		= s_\pi^* (b_1 \tensor \one)s_\pi s_\pi^*(b_2 \tensor \one) s_\pi 
		= s_\pi^* (b_1 \tensor \one)(b_2 \tensor \one) s_\pi
		= \gamma_\pi(b_1 b_2) .
		\qedhere
	\end{equation*}
\end{proof}

\begin{rmk}
	\label{rmk:gamma_unique}
	We would like to point out that, since $s_\pi$ is non-degenerated, the map $\gamma_\pi:\aB \to \aB \tensor \End(V_\pi)$ is uniquely determined by the property $(b \tensor \one) s_\pi = s_\pi \gamma_\pi(b)$ for all $b \in \aB$.
\end{rmk}

In the following we will frequently deal with tensor products in which precisely one factor is $\aB$. In this cases we will allow more flexibility for the position of the factor $\aB$, that is, we reshuffle the tensor factor is such a way that $\aB$ is at a convenient position, usually the first factor, and the other factors are kept in order. 

Coming back to the dynamical system, the multiplicative structure among the sets $A_2(\pi)$ for different representations can be phrased in terms of the elements $s_\pi$, too. For two irreducible representations $\pi, \rho \in \hat G$ consider the multiplication map
\begin{equation}\label{multiplication maps}
\begin{gathered}
	m_{\pi, \rho}: A_2(\pi) \tensor_\aB A_2(\rho) \longrightarrow A_2(\pi \tensor \rho) 
	\subseteq \aA \tensor \End(V_\pi) \tensor \End(V_\rho),
	\\
	m_{\pi,\rho}(x \tensor y)  := x_{12} \, y_{13}.
\end{gathered}
\end{equation}

This is a module map for the right action of $\aB \tensor \End(V_\pi) \tensor \End(V_\rho)$ on domain and codomain. Therefore, it is uniquely determined by the element $m_{\pi, \rho}(s_\pi, s_\rho)$. For this element there is a unique element $\omega(\pi, \rho) \in \aB \tensor \End(V_\pi \tensor V_\rho)$ with $m_{\pi, \rho}(s_\pi, s_\rho) = s_{\pi \tensor \rho} \cdot \omega(\pi, \rho)$. In fact, $\omega(\pi,\rho)$ is the isometry given by
\begin{equation}
	\label{eq:u_concrete}
	\omega(\pi,\rho) = s_{\pi \tensor \rho}^* \, (s_\pi)_{12} \, (s_\rho)_{13}.
\end{equation}
With this element the multiplication map can be written as
\begin{equation}
	\label{eq:mult_map}
	m_{\pi, \rho}(s_\pi x \tensor s_\rho y) 
	= s_{\pi \tensor \rho} \cdot \omega(\pi,\rho) \cdot (\id_\pi \tensor \gamma_\rho)(x) \cdot (\one_\pi \tensor y)
\end{equation}
for all $x \in \aB \tensor \End(V_\pi)$, $y \in \aB \tensor \End(V_\rho)$.

We want to classify and characterize weakly cleft C\Star dynamical systems in terms of the \Star homomorphisms $\gamma_\pi$ and the isometries $\omega(\pi,\rho)$. For this purpose, we fix a group $G$ and a C\Star algebra $\aB$ and we consider pairs $(\gamma, \omega)$ consisting of two families $\gamma = (\gamma_\pi)_{\pi \in \hat G}$ and $\omega = \bigl( \omega(\pi,\rho) \bigr)_{\pi, \rho \in \hat G}$ where
\begin{enumerate}
\item 
	for each $\pi \in \hat G$, we have a unital \Star homomorphism $\gamma_\pi:\aB \to \aB \tensor \End(V_\pi)$ and
\item
	for each $\pi, \rho \in \hat G$, we have an isometry $\omega(\pi,\rho) \in \aB \tensor \End(V_\pi \tensor V_\rho)$.
\end{enumerate}

\pagebreak[3]
\begin{defn}
	\label{defn:factor_sys}
	\begin{enumerate}
	\item 
		The pair $(\gamma, \omega)$ is called a \emph{factor system} for $(G,\aB)$ if $\omega(1,1) = \one_\aB$ and the family satisfies
		\begin{align}
			\label{eq:gamma}
			\omega(\pi,\rho) \cdot (\id_\pi \tensor \gamma_\rho) \bigl( \gamma_\pi(b) \bigr) 
			&= \gamma_{\pi \tensor \rho}(b) \cdot \omega(\pi,\rho),
			\\
			\label{eq:u}
			\bigl( \one_\pi \tensor \omega(\rho,\sigma) \bigr) \cdot (\id_{\pi \tensor \rho} \tensor \gamma_\sigma)\bigl( \omega(\pi,\rho) \bigr)^*
			&=
			\omega(\pi, \rho \tensor \sigma)^* \cdot \omega(\pi \tensor \rho, \sigma)
		\end{align}
		for all $\pi, \rho, \sigma \in \hat G$ and $b \in \aB$.
	\item
		Two factor systems $(\gamma, \omega)$ and $(\gamma', \omega')$ for $(G,\aB)$ are called \emph{conjugated} if there is a family $v = (v_\pi)_{\pi \in \hat G}$ of unitaries $v_\pi \in \aB \tensor \End(V_\pi)$ such that
		\begin{align*}
			\gamma_\pi' 
			&= \Ad[v_\pi^*] \circ \gamma_\pi,
			\\
			v_{\pi\tensor  \rho} \cdot \omega'(\pi,\rho) 
			&= \omega(\pi,\rho) \cdot (\id_\pi \tensor \gamma_\rho)(v_\pi) \cdot v_\rho .
		\end{align*}
	\end{enumerate}
\end{defn}
Equations \eqref{eq:gamma} and \eqref{eq:u} are twisted versions of the equations for a coaction and for a 2-cocycle. For this reason, we refer to condition \eqref{eq:gamma} as the \emph{coaction condition} and to condition~\eqref{eq:u} as the \emph{cocycle condition}. Immediate examples of factor systems are accordingly given a by coaction with a trivial cocycle or by a 2-cocycle with a trivial coaction (see Remark~\ref{rmk:factor_sys}). 

\begin{rmk}
	\label{rmk:factor_sys}
	\begin{enumerate}
	\item 
		The normalization condition $\omega(1,1) = \one_\aB$ can always be achieved by passing to a normalized, conjugated system (in the straightforwardly generalized sense). Together with the coaction and cocycle condition the normalization implies $\gamma_1 = \id_\aB$ and $\omega(\pi,1) = \one = \omega(1,\pi)$ for all $\pi \in \hat G$. 
	\item
		It is worth mention that we may rephrase things in terms of the group C\Star algebra $C^*(G)$, more precisely its multiplier algebra $MC^*(G)$. The family $(\gamma_\pi)_{\pi \in \hat G}$ may be equivalently written as a single unital \Star homomorphism $\gamma:\aB \to \aB \tensor MC^*(G)$ and the family $( \omega(\pi,\rho))_{\pi, \rho \in \hat G}$ as an isometry \mbox{$\omega \in \aB \tensor MC^*(G \times G)$}. Then Equations~\eqref{eq:gamma} and ~\eqref{eq:u} can be casted in the form
		\begin{align*}
			\omega \cdot \bigl( (\id \tensor \gamma) \circ \gamma \bigr)(b) 
			&= (\delta \circ \gamma)(b) \cdot \omega,
			\\
			(\one \tensor \omega) \cdot (\id \tensor \gamma)(\omega^*) 
			&=
			(\id \tensor \delta)(\omega^*) \cdot (\delta \tensor \id)(\omega) ,
		\end{align*}
		where $\delta:MC^*(G) \to MC^*(G \times G)$ denotes the usual comultiplication. This formulation is used for extending the theory to compact quantum groups. If~$\omega$ is unitary, the pair $(\gamma, \omega)$ is sometimes called a cocycle action (see \eg \cite{VaVa03}). 
	\end{enumerate}
\end{rmk}

The following statement summarizes that the construction from the beginning of this section indeed provides examples of factor systems.
\begin{lemma}
	Let $(\aA, G, \alpha)$ be a weakly cleft C\Star dynamical system and let $\aB$ denote its fixed point algebra. Furthermore, let $s = (s_\pi)_{\pi \in \hat G}$ be a family of non-degenerate isometries $s_\pi \in A_2(\pi)$ with $s_1 = \one_\aB$. If $\gamma = (\gamma_\pi)_{\pi \in \hat G}$ and $\omega = \bigl(\omega(\pi,\rho) \bigr)_{\pi, \rho \in \hat G}$ are the associated families of homomorphisms and unitaries given by Equations~\eqref{eq:gamma_concrete} and ~\eqref{eq:u_concrete}, respectively, then the pair $(\gamma, \omega)$ is a factor system.
\end{lemma}
\begin{proof}
	In order to show the coaction condition \eqref{eq:gamma}, we recall that we have $s_\pi \gamma_\pi(b) = (b \tensor \one) s_\pi$ for every $b \in \aB$ (see Remark~\ref{rmk:gamma_unique}). Successively applying this relation and its starred version then yields Equation~\eqref{eq:gamma}, \ie, in $\aB \tensor \End(V_\pi) \tensor \End(V_\rho)$ we obtain
	\begin{align*}
		\omega(\pi, \rho) \cdot \gamma_\rho \bigl( \gamma_\pi(b) \bigr) 
		&= s_{\pi \tensor \rho}^* s_\pi s_\rho \gamma_\rho \bigl( \gamma_\pi(b) \bigr) 
		= s_{\pi \tensor\rho}^* b s_\pi s_\rho 
		\\
		&= \gamma_{\pi \tensor \rho}(b) s_{\pi \tensor \rho}^* s_\pi s_\rho 
		=  \gamma_{\pi \tensor \rho}(b) \cdot \omega(\pi, \rho) 
	\end{align*}
	for all $\pi, \rho \in \hat G$ and $b \in \aB$.  
	In order to verify the cocycle condition \eqref{eq:u} let us first consider its left hand side 
	\begin{align*}
		\omega(\rho, \sigma) \cdot \gamma_\sigma \bigl( \omega(\pi, \rho) \bigr)^*
		&=
		s_{\rho \tensor \sigma}^* s_\rho s_\sigma \cdot s_\sigma^* (s_{\pi \tensor \rho}^* s_\pi s_\rho)^* s_\sigma
		=
		s_{\rho \tensor \sigma}^* s_\rho s_\sigma s_\sigma^* s_\rho^* s_\pi^* s_{\pi \tensor \rho} s_\sigma
	\end{align*}
	The rightmost product $s_\rho^* s_\pi^* s_{\pi \tensor \rho} s_\sigma$ lies in $\End(V_\pi \tensor V_\rho) \tensor A_2(\sigma)$. Therefore, the non-degeneracy of $s_\sigma$ allows us to cancel the factor $s_\sigma s_\sigma^*$. Similarly, the element $s_\pi^* s_{\pi \tensor \rho}$ lies in $\End(V_\pi) \tensor A_2(\rho)$, which allows us to cancel even further to obtain
	\begin{align*}
		\omega(\rho, \sigma)  \cdot \gamma_\sigma \bigl( \omega(\pi,\rho) \bigr)^*
		&=
		s_{\rho \tensor \sigma}^* s_\pi^* s_{\pi \tensor \rho} s_\sigma.
	\end{align*}
	For the right hand side of the cocycle condition, the non-degeneracy of $s_{\pi \tensor \rho \tensor \sigma}$ likewise implies
	\begin{align*}
		\omega(\pi, \rho \tensor \sigma)^* \cdot \omega(\pi \tensor \rho, \sigma) 
		&=
		s_{\rho \tensor \sigma}^* s_\pi^* s_{\pi \tensor \rho \tensor \sigma} s_{\pi \tensor \rho \tensor \sigma}^* s_{\pi \tensor \rho} s_\sigma
		= s_{\rho \tensor \sigma}^* s_\pi^* s_{\pi \tensor \rho} s_\sigma.
	\end{align*}
	Comparing with the simplification of the left side then yields the cocycle condition for all $\pi, \rho, \sigma \in \hat G$.
\end{proof}

The next result states that the weakly cleft C\Star dynamical systems are uniquely determined by their factor systems up to equivalence. 
\pagebreak[3]
\begin{thm}
	\label{thm:factor_sys_concrete}
	Let $(\aA, G, \alpha)$ and $(\aA', G, \alpha')$ be weakly cleft C\Star dynamical systems with the same fixed point algebra $\aB$ and let $(\gamma,\omega)$ and $(\gamma', \omega')$ be associated factor systems, respectively. 
	Then the following statements are equivalent:
	\begin{equivalence}
	\item	\label{en:dyn_sys_equiv}
		The dynamical systems $(\aA, G, \alpha)$ and $(\aA', G, \alpha')$ are equivalent.
	\item	\label{en:fac_sys_conj}
		The factor systems $(\gamma, \omega)$ and $(\gamma', \omega')$ are conjugated.
	\end{equivalence}
\end{thm}
\begin{proof}
	As a distinction we add a prime to all notions referring to $(\aA', G, \alpha')$.
	\begin{enumerate}
	\item 
		To prove that \ref{en:dyn_sys_equiv} implies \ref{en:fac_sys_conj} it suffice to show that for the same dynamical system $(\aA, G, \alpha)$ different choices of non-degenerate isometries $s_\pi \in A_2(\pi)$, $\pi \in \hat G$, lead to conjugated factor systems. For this purpose let $s_\pi$ and $s_\pi'$,  $\pi \in \hat G$, be two such choices and let us denote by $(\gamma, \omega)$ and $(\gamma', \omega')$ the associated factor systems, respectively. Consider first a fixed representation $\pi \in \hat G$. By Lemma~\ref{lem:cleft_isom} there are unique elements $v_\pi^{}, v_\pi' \in \aB \tensor \End(V_\pi)$ with $s_\pi' = s_\pi v_\pi$ and $s_\pi = s_\pi' v_\pi'$. Uniqueness implies $v_\pi^{} v_\pi' = \one = v_\pi' v_\pi^{}$ and, since $s_\pi$ is a isometry, we also have $v_\pi^* v_\pi^{} = v_\pi^* s_\pi^* s_\pi^{} v_\pi^{} = (s_\pi')^* (s_\pi') = \one$.  Hence $v_\pi$ and $v_\pi'$ are unitaries with $v_\pi' = v_\pi^*$. For the \Star homomorphisms of the factor systems we therefore obtain
		\begin{equation*}
			\gamma_\pi'(b) 
			= (s_\pi')^* (b\tensor \one) s_\pi' 
			= v_\pi^* s_\pi^* (b\tensor \one) s_\pi v_\pi 
			= v_\pi^* \gamma_\pi(b) v_\pi,
		\end{equation*}
		for every $\pi \in \hat G$ and $b \in \aB$. For the isometries of the factor systems we may use the non-degeneracy of $s_\rho$ and the fact that $v_\pi s_\rho v_\rho \in \End(V_\pi) \tensor A_2(\rho)$ to conclude for all $\pi, \rho \in \hat G$:
		\begin{align*}
			v_{\pi \tensor \rho} \, \omega'(\pi, \rho) 
			&= v_{\pi \tensor \rho} (s_{\pi \tensor  \rho} v_{\pi \tensor \rho})^* s_\pi v_\pi s_\rho v_\rho
			= s_{\pi \tensor \rho}^* s_\pi (s_\rho^{} s_\rho^*) v_\pi s_\rho v_\rho
			\\
			&= s_{\pi \tensor \rho}^* s_\pi s_\rho \gamma_\rho(v_\pi) v_\rho
			= \omega(\pi, \rho) \,  \gamma_\rho(v_\pi) \, v_\rho.
		\end{align*}
	\item
		For the converse implication, \ref{en:fac_sys_conj} $\Rightarrow$ \ref{en:dyn_sys_equiv}, let $s_\pi \in A_2(\pi)$ and $s_\pi' \in A_2'(\pi)$, $\pi \in \hat G$, by non-degenerate isometries with associated factor systems $(\gamma, \omega)$ and $(\gamma', \omega')$, respectively. Furthermore, let $v_\pi$, $\pi \in \hat G$, be a family of unitaries realizing the conjugation of the factor systems as in Definition~\ref{defn:factor_sys}. For every representation $\pi$ of $G$ the map
		\begin{equation*}
			\varphi_\pi: A_2'(\pi) \to A_2(\pi), 
			\quad
			s_\pi' x  \mapsto s_\pi v_\pi x
		\end{equation*}
		for all $x \in \aB \tensor \End(V_\pi)$ is well-defined by Lemma~\ref{lem:cleft_isom}. Moreover, it is straightforward to check that $\varphi_\pi$ is a unitary map between the two right Hilbert $\aB$-modules. Since $L^2(\aA) = \overline{\bigoplus}_{\pi \in \hat G} A_2(\pi)$ and likewise for $\aA'$, taking direct sums yields a unitary map
		\begin{equation*}
			V:L^2(\aA') \to L^2(\aA), 
			\quad
			V := \bigoplus\nolimits_{\pi \in \hat G} \varphi_\pi .
		\end{equation*}
		Furthermore, the maps $\varphi_\pi$, $\pi \in \hat G$, intertwine with the multiplication maps, that is, for all $\pi, \rho \in \hat G$ we have 
		\begin{equation*}
			m_{\pi, \rho} \bigl( \varphi_\pi(s_\pi x) \tensor \varphi_\rho(s_\rho y) \bigr) = \varphi_{\pi \tensor \rho} \bigl( m_{\pi, \rho}'(s_\pi x \tensor s_\pi y) \bigr)
		\end{equation*}
		for all $x \in \aB \tensor \End(V_\pi)$ and $y \in \aB \tensor \End(V_\rho)$. Together with Equation~\eqref{eq:right_mult} this shows that the homomorphism
		\begin{equation*}
			\Phi(x) := Vx V^\star
		\end{equation*} 
		maps $\aA' \subseteq \End \bigl( L^2(\aA') \bigr)$ into $\aA \subseteq \End \bigl( L^2(\aA) \bigr)$ and hence may be restricted to an injective \Star homomorphism $\Phi:\aA' \to \aA$. Exchanging the role of $s_\pi$ and $s_\pi'$ shows that $\Phi$ is in fact a \Star isomorphism. Obviously, we have $\varphi_\pi(x \pi_g^*) = \varphi_\pi(x) \pi_g^*$ for all $x \in A_2(\pi)$ and $g \in G$. It follows that $V$ intertwines the $G$-action on $L^2(\aA)$ and $L^2(\aA')$, that is, we have $V U_g = U_g' V$ for all $g \in G$. Consequently, $\Phi$ intertwines $\alpha_g = \Ad[U_g]$ on $\aA$ and $\alpha'_g= \Ad[U_g']$ on $\aA'$.
		\qedhere
	\end{enumerate}
\end{proof}

The following lemma rephrases freeness in terms of the multiplication maps defined in \eqref{multiplication maps}. As a consequence we find that cleft dynamical systems are characterized factor systems where the isometries $\omega(\pi, \rho)$ are in fact unitaries.

\begin{lemma}\label{lem:characterization freeness}
	A C\Star dynamical system $(\aA, G, \alpha)$ is free if and only if for all $\pi \in \hat G$ the multiplication map
	\begin{equation*}
		m_{\pi, \bar\pi}:A_2(\pi) \tensor_\aB A_2(\bar\pi) \to A_2(\pi \tensor \bar\pi),
		\quad
		m_{\pi, \bar\pi}(x \tensor \bar y) := x_{12} \cdot \bar y_{13}
	\end{equation*}
	has dense range or, equivalently, is surjective.
\end{lemma}
\begin{proof}
	First we note that $m_{\pi, \bar\pi}$ is an isometry of correspondences over $\aB$ and hence it is surjective if and only if it has dense range. Let us fix a finite-dimensional representation $(\pi, V)$ of $G$ and denote by $d$ its dimension. For sake of a convenient notation we fix a basis of $V$ and write elements $x \in \aA \tensor \End(V)$ as matrices $x = (x_{i,j})_{1 \le i,j \le d}$ with entries in~$\aA$. Likewise we write elements of $\aA \tensor \End(\bar V)$ and $\aA \tensor \End(V \tensor \bar V)$ as matrices with respect to the dual basis on $\bar V$ and the product basis, respectively. 
	A straightforward computation shows that the transpose map $A_2(\bar\pi) \to A_2(\pi)^*$, $x = (x_{i,j})_{i,j} \mapsto x^t := (x_{j,i})_{i,j}$ is a linear bijection. Moreover, similar computations show that the map  
	\begin{equation*}
		\varphi:A_2(\pi \tensor \bar\pi) \to C(\pi) \tensor \End(V),
		\quad
		\varphi(x)_{(i,j), (k,\ell)} := x_{(i,k), (j,\ell)}
	\end{equation*}
	is a linear bijection. Now consider the composition $\psi := \varphi \circ m_{\pi, \bar\pi}$, which takes the concrete form
	\begin{equation*}
		\psi(x \tensor \bar y)_{(i,j), (k,\ell)} = x_{i,j} \, \bar y_{k,\ell}
	\end{equation*}
	for $x \in A_2(\pi)$ and $\bar y \in A_2(\bar\pi)$. Since $A_2(\pi)$ and $A_2(\bar\pi)$ are right $\End(V)$- and $\End(\bar V)$-modules, respectively, the range of $\psi$ is a bimodule for $\one_{C(\pi)} \tensor \End(V)$. It follows that the range of $\psi$ is of the form $J \tensor \End(V)$ for some subspace $J \subseteq C(\pi)$. 
	Furthermore, for elements $x,y \in A_2(\pi)$ we may put $\bar y := (y^*)^t$ and find 
	\begin{align*}
		(\id_{C(\pi)} \tensor \Tr) \bigl (\psi(x \tensor \bar y) \bigr)
			&= \biggl(\sum_{\ell=1}^d x_{i,\ell} \, y_{\ell,j}^* \biggr)_{i,j}
			= x y^*,
	\end{align*}
	which shows that $J = A_2(\pi) A_2(\pi)^*$.
	We conclude that the multiplication map $m_{\pi, \bar\pi}$ is surjective if and only if $\psi$ has full range if and only if $C(\pi) = J = A_2(\pi) A_2(\pi)^*$.
\end{proof}

\begin{thm}
	\label{thm:free_cleft}
	For a weakly cleft C\Star dynamical system $(\aA, G, \alpha)$ the following statements are equivalent:
	\begin{equivalence}
	\item
		The dynamical system is cleft or, equivalently, free.
	\item
		For some factor system $(\gamma, \omega)$---and hence for all factor systems---all elements $\omega(\pi,\rho)$ for $\pi, \rho \in \hat G$ are unitary.
	\end{equivalence}
\end{thm}
\begin{proof}
	By Lemma~\ref{lem:cleft_n_sat}, for a cleft system we may choose unitary elements $s_\pi$ in $A_2(\pi)$. Then it follows that the elements $\omega(\pi,\rho) = s_{\pi \tensor \rho}^* s_\pi^{} s_\rho^{}$ ($\pi, \rho \in \hat G$) of the corresponding factor system are unitary, too, which proves one implication. 
	For the converse implication we take advantage of Lemma \ref{lem:characterization freeness}. Indeed, suppose that we have non-degenerated isometries \mbox{$s_\pi \in A_2(\pi)$} for each $\pi \in \hat G$ such that the elements $\omega(\pi,\rho) = s_{\pi \tensor \rho}^* s_\pi s_\rho$ of the corresponding factor system are unitaries for all $\pi, \rho \in \hat G$. 
	Then $A_2(\pi \tensor \rho)$ is given by 
	\begin{align*}
		A_2(\pi \tensor \rho) 
		&= s_{\pi \tensor \rho} \cdot  \aB \tensor \End(V_\pi \tensor V_\rho) 
		\\
		&= s_{\pi \tensor \rho} \, \omega(\pi,\rho) \cdot  \aB \tensor \End(V_\pi \tensor V_\rho) 
		= s_\pi s_\rho \cdot \aB \tensor \End(V_\pi \tensor V_\rho) 
		\\
		&= s_\pi \cdot \bigl[ \one_\aB \tensor \End(V_\pi) \tensor \one_\rho \bigr] \cdot s_\rho \cdot \bigl[ \aB \tensor \one_\pi \tensor \End(V_\rho) \bigr]
	\end{align*}
	Since $s_\pi \cdot \bigl( \one_\aB \tensor \End(V_\pi) \bigl) \subseteq A_2(\pi)$ and $s_\rho \cdot \bigl( \aB \tensor \End(V_\rho) \bigr) = A_2(\rho)$, we conclude that $m_{\pi, \rho}$ is surjective and hence the dynamical system is free. 
\end{proof}

\begin{cor}
	\label{cor:comm=>free}
	Every weakly cleft C\Star dynamical system $(\aA, G, \alpha)$ with commutative or finite-dimensional fixed point algebra $\aB$ is cleft. 
\end{cor}
\begin{proof}
	Under the hypothesis on $\aB$ an isometry $\omega \in \aB \tensor \End(V)$ for finite-dimensional $V$ is automatically unitary.
\end{proof}

In Section~\ref{sec:construction} we will discuss whether the converse of Theorem~\ref{thm:factor_sys_concrete} holds, that is, whether every factor system gives rise to a C\Star dynamical system. We postpone this problem for the moment and continue to investigate the relation between the form of the factor system and the type of the dynamical system. First, let us characterize classical systems.

\begin{cor}
	\label{cor:commutative_syst}
	Let $(\aA, G, \alpha)$ be a cleft C\Star dynamical system with a commutative fixed point algebra~$\aB$ and let $(\gamma, \omega)$ be an arbitrary associated factor system. Then $\aA$ is commutative if and only if 
	\begin{align*}
		\gamma_\pi(b) &= b \tensor \one_\pi,
		&
		&\text{and}
		&
		\omega(\rho, \pi) &= \sigma \bigl( \omega(\pi,\rho) \bigr)
	\end{align*}
	for all $\pi, \rho \in \hat G$ and $b \in \aB$, where $\sigma$ denotes the tensor flip of $\End(V_\pi) \tensor \End(V_\rho)$.
\end{cor}
\begin{proof}
	First suppose that $\aA$ is commutative. Let $s_\pi \in A_2(\pi)$, $\pi \in \hat G$, be a family of non-degenerated isometries and denote by $(\gamma, u)$ the associated factor system. Since $\aA$ is commutative, every element $b \tensor \one$ commutes with $s_\pi$ in $\aA \tensor \End( V_\pi)$. It follows that $\gamma_\pi(b) = s_\pi^*(b \tensor \one)s_\pi = b \tensor \one$ for all $b \in \aB$. Likewise the elements $\smash{(s_\pi)_{12}}$ and $\smash{(s_\rho)_{13}}$ commute in $\aA \tensor \End(V_\pi) \tensor \End(V_\rho)$. It follows that
	\begin{equation*}
		\sigma \bigl( \omega(\rho, \pi) \bigr)
		= \sigma \bigl( s_{\rho \tensor \pi}^* s_\rho s_\pi \bigr) 
		= s_{\pi \tensor \rho}^* \, s_\pi \, s_\rho 
		=  \omega(\pi, \rho) .
	\end{equation*}
	Conversely, let $s_\pi \in A_2(\pi)$, $\pi \in \hat G$, be a family of non-degenerate isometries and suppose the asserted condition on the corresponding factor system holds. Consider the family of multiplication maps 
	\begin{gather*}
		m_{\pi, \rho}:A_2(\pi) \tensor_\aB A_2(\rho) \to A_2(\pi \tensor \rho),
		\quad
		m_{\pi, \rho}(x \tensor y) = x_{12} \, y_{13}.
	\end{gather*}
	Looking at Equation~\eqref{eq:mult_map}, we have for all $x \in \aB \tensor \End(V_\pi)$ and $y \in \aB \tensor \End(V_\rho)$:
	\begin{align*}
		\sigma \bigl( m_{\rho, \pi}(s_\rho y \tensor s_\pi x) \bigr) 
		&= \sigma \bigl( s_{\rho \tensor \pi} \, \omega(\rho, \pi) \, y_{12} \, x_{13} \bigr)
		= s_{\pi \tensor \rho} \, \omega(\pi, \rho) \, x_{12} \, y_{13}
		\\
		&= m_{\pi, \rho}(s_\pi x \tensor s_\rho y),
	\end{align*}
	because $x_{12}$ and $y_{13}$ commute in $\aB \tensor \End(V_\pi) \tensor \End(V_\rho)$. From the general construction of Section~\ref{sec:decomp} we may then deduce that $\aA$ is commutative.
\end{proof}

\begin{rmk}	\label{rmk:centre}
	\begin{enumerate}
	\item
		Corollary~\ref{cor:commutative_syst} provides a classification of cleft topological principal bundles. 
	\item 
		Similar arguments as in the proof of Corollary~\ref{cor:commutative_syst} show the following more general statement for a weakly cleft system: The center of $\aB$ is contained in the center of $\aA$ if and only if for every factor system $(\gamma, \omega)$ the homomorphisms act trivially on the center of $\aB$, \ie, $\gamma_\pi(z) = z \tensor \one$ for every central element $z \in \aB$ and $\pi \in \hat G$.
	\end{enumerate}
\end{rmk}

For a given group $G$ and a given fixed point algebra $\aB$ there clearly seems to be a trivial dynamical system, namely $(\aB \tensor C(G), G, \id \tensor r)$. From our discussions in Example~\ref{expl:trivial_bundle} we immediately derive from Theorem~\ref{thm:factor_sys_concrete} that for a weakly cleft dynamical system $(\aA, G, \alpha)$ the following statements are equivalent:
\begin{equivalence}
\item
	The dynamical system $(\aA, G, \alpha)$ is equivalent to $\bigl( \aB \tensor C(G), G, \id \tensor r \bigr)$.
\item
	Every factor system of $(\aA, G, \alpha)$ is conjugated to the factor system $(\gamma, \omega)$ given by
	\begin{align*}
		\gamma_\pi(b) &:= b \tensor \one
		&
		\omega(\pi, \rho) &:= \one,
		&
		&\pi, \rho \in \hat G, \; b \in \aB
	\end{align*}
\end{equivalence}

A next simple class of C\Star dynamical systems are those which are essentially ergodic actions, that is, the dynamical system $(\aA, G, \alpha)$ arises from an ergodic dynamical system $(\aA_0, G, \alpha_0)$ by tensoring with $\aB$. 

\pagebreak[3]
\begin{cor}
	\label{cor:essentially_ergodic}
	For a C\Star dynamical system $(\aA, G, \alpha)$ the following statements are equivalent:
	\begin{equivalence}
	\item
		$(\aA, G, \alpha)$ is isomorphic to $(\aB \tensor \aA_0, G, \id_\aB \tensor \alpha_0)$ with an ergodic cleft C\Star dynamical system $(\aA_0, G, \alpha_0)$.
	\item
		$(\aA, G, \alpha)$ is (weakly) cleft and admits a factor system of the form
		\begin{align*}
			\gamma_\pi(b) &= b \tensor \one,
			&
			\omega(\pi, \rho) &\in \one_\aB \tensor \End(V_\pi \tensor V_\rho),
			&
			&\pi, \rho \in \hat G, \; b \in \aB.
		\end{align*}
	\end{equivalence}
\end{cor}
\begin{proof}
	For the system $(\aB \tensor \aA_0, G, \id_\aB \tensor \alpha_0)$ the generalized isotypic component of $\pi \in \hat G$ is of the form
	\begin{equation*}
		A_2(\pi) = \aB \tensor A_2^{(0)}(\pi),
	\end{equation*}
	where $A_2^{(0)}(\pi) \subseteq \aA_0 \tensor \End(V_\pi)$ denotes the generalized isotypic component of $(\aA_0, G, \alpha_0)$. Since $\alpha_0$ is a cleft action, for each $\pi \in \hat G$ we find a unitary element $s_\pi^{(0)} \in A_2^{(0)}(\pi)$ and obtain a unitary $s_\pi \in A_2(\pi)$ by putting $s_\pi := \one_\aB \tensor s_\pi^{(0)}$. Then the corresponding factor system defined by Equations~\eqref{eq:gamma_concrete} and~\eqref{eq:u_concrete} is of the asserted form. 

	Conversely, suppose that a factor system $(\gamma, \omega)$ for $(\aA, G, \alpha)$ has the asserted form. Then $\omega(\pi, \rho)$ is a unitary for each $\pi, \rho \in \hat G$ and the cocycle condition \eqref{eq:u} states that $\omega =  \bigl( \omega(\pi, \rho) \bigr)_{\pi, \rho}$ forms a unitary 2-cocycle for $G$ in the usual sense. By \cite{Wass89} this 2-cocycle provides an ergodic cleft C\Star dynamical system $(\aA_0, G, \alpha_0)$ whose factor system is given by $\omega$ together with the trivial homomorphisms. (This may also be derived from the construction presented in Section~\ref{sec:construction}.) By the arguments in the first part of the proof the dynamical system $(\aB \tensor \aA_0, G, \id_\aB \tensor \alpha_0)$ then has the factor system $(\gamma, \omega)$ and, by Theorem~\ref{thm:factor_sys_concrete}, it is isomorphic to $(\aA, G, \alpha)$.
\end{proof}

\begin{expl}
	Let $\aB = \End(\hilb H)$ for some Hilbert space $\hilb H$. We claim that up to isomorphism every cleft dynamical system with fixed point algebra $\End(\hilb H)$ is  of the form $(\End(\hilb H) \tensor \aA_0, G, \id \tensor \alpha_0)$ with an ergodic cleft action $(\aA_0, G, \alpha_0)$. 
	This can be proved directly by looking at minimal projections $p \in \End(\hilb H)$ and the corresponding restricted dynamical system on the algebra~$p\aA p$. Alternatively, we may consider an arbitrary factor system $(\gamma, \omega)$ for $(\aA, G, \alpha)$. Then for each representation $\pi \in \hat G$ the \Star homomorphism $\gamma_\pi:\End(\hilb H) \to \End(\hilb H) \tensor \End(V_\pi)$ is necessarily of the form $\gamma_\pi(b) = v_\pi^*(b \tensor \one)v_\pi$, $b \in \aB$, for some unitary $v_\pi \in \End(\hilb H) \tensor \End(V_\pi)$. Passing from $(\gamma, \omega)$ to a conjugated factor system, we~may without loss of generality assume $\gamma_\pi(b) = b \tensor \one$ for all $b \in \aB$. Then the coaction condition states that for all $\pi, \rho \in \hat G$ the element $\omega(\pi, \rho) \in \End(\hilb H \tensor V_\pi \tensor V_\rho)$ commutes with $\End(\hilb H) \tensor \one_{\pi \tensor \rho}$ and hence lies in $\one \tensor \End(V_\pi \tensor V_\rho)$. Finally, Corollary~\ref{cor:essentially_ergodic} proves the claim.
\end{expl}


\pagebreak[3]
\begin{cor}
	\label{cor:ergodic_subsyst}
	For a C\Star dynamical system $(\aA, G, \alpha)$ and a unital C\Star sub\-al\-ge\-bra $\aB_0 \subseteq \aB$ the following statements are equivalent:
	\begin{equivalence}
	\item	\label{en:inv_subalg}
		There is an $\alpha$-invariant C\Star subalgebra $\aB_0 \subseteq \aA_0 \subseteq \aA$ such that the restricted system $(\aA_0, G, \alpha|_{\aA_0})$ is cleft with fixed point algebra $\aB_0$.
	\item	\label{en:factor_sys_B0}
		There is a factor system $(\gamma,\omega)$ of $\aA$ with unitary elements $\omega(\pi, \rho) \in \aB_0 \tensor \End(V_\pi \tensor V_\rho)$ and $\gamma_\pi(\aB_0) \subseteq \aB_0 \tensor \End(V_\pi)$ for all $\pi, \rho \in \hat G$.
	\end{equivalence}
	In this case $\aA_0$ can be chosen to have the factor system $\bigl( \gamma_\pi|_{\aB_0}, \omega(\pi, \rho) \bigr)_{\pi, \rho \in \hat G}$ and $\aA$ is generated by $\aA_0$ and $\aB$.
\end{cor}
To have a concise statement we claim equivalence here but at this point we prove only one implication. The converse implication will be deduced later as Corollary~\ref{cor:remaining}.
\begin{proof}[Proof of \ref{en:inv_subalg} $\Rightarrow$ \ref{en:factor_sys_B0}]
	For an $\alpha$-invariant subalgebra $\aA_0 \subseteq \aA$ its generalized isotypic component of $\pi \in \hat G$, denoted by $A_2^{(0)}(\pi)$, is contained in the generalized isotypic component $A_2(\pi)$ of $\aA$. If the action on $\aA_0$ is cleft, then for each $\pi \in \hat G$ we find a unitary element $s_\pi \in A_2^{(0)}(\pi)$. The elements $s_\pi$, $\pi \in \hat G$, give rise to a factor system for $\aA_0$ and to a factor system for $\aA$. Both factor systems share the same unitaries $\omega(\pi,\rho) = s_{\pi \tensor \rho}^* s_\pi^{} s_\rho^{} \in \aB_0 \tensor \End(V_\pi \tensor V_\rho)$ for all $\pi, \rho \in \hat G$. The \Star homomorphisms of the two factor systems are both given by $\gamma_\pi(b) = s_\pi^* (b \tensor \one) s_\pi^{}$ for $\pi \in \hat G$, that is, they only differ by their domains $\aB_0$ and $\aB$, respectively, and the corresponding codomains.
\end{proof}

\begin{expl}
	\label{expl:SU_2}
	We would like to present a simple example that is not a tensor product with a free ergodic action.  For this purpose consider the group $G = \SU_2$ and the commutative C\Star algebra $C(X)$ for an arbitrary compact space $X$ on which we fix a non-trivial continuous reflection $h:X \to X$, $h \circ h = \id_X$. The group $\SU_2$ has up to equivalence for every dimension precisely one irreducible representation. So we may identify $\hat G$ with the natural numbers starting with $V_0$ denoting the trivial representation and $V_1 = \C^2$ as the standard representation of $\SU_2$.
	As a factor system $(\gamma, \omega)$ we choose the coaction \mbox{$\gamma_n:\Cont(X) \to \Cont(X) \tensor \End(V_n)$}, $n \in \N$, given by 
	\begin{equation*}
		\gamma_n(b) := \begin{cases}
			b \tensor \one &\text{if $n$ is even, } \\
			(b \circ h) \tensor \one &\text{if $n$ is odd}
		\end{cases}
	\end{equation*}
	accompanied by the trivial 2-cocycle $\omega(n,m) := \one$ for $n,m \in \N$. The coaction condition for $\gamma = (\gamma_n)_{n\in \N}$ is easily verified, \eg, using the Clebsch-Gordan formula for $\SU_2$.
	Suppose for a moment that there exists a C\Star dynamical system $(\aA, G, \alpha)$ with the factor system $(\gamma,\omega)$. (This will be proven in Section~\ref{sec:construction}.) It follows from Corollary~\ref{cor:ergodic_subsyst} that there is an $\alpha$-invariant unital subalgebra $\aA_0 \subseteq \aA$ such that the restricted system is equivalent to the trivial system $\bigl( C(\SU_2), G, r \bigr)$ and the algebra $\aA$ is generated by $\aA_0 = C(\SU_2)$ and $\aB = C(X)$. However, these two algebras satisfy non-trivial commutation relations.  Namely, we may collect all odd and even isotypic components and split $\aA$ into its odd and even part, that is, we call $x \in \aA$ even if $\alpha_{-1}(x) = x$ and odd if $\alpha_{-1}(x) = (-x)$. Recall that for each $(\pi,V) \in \hat G$ the isotypic component $A(\pi)$, as correspondence over~$\aB$, is isomorphic to $\aB \tensor \End(V)$ with the usual right right multiplication and the left multiplication $b \,.\, x = \gamma_\pi(b) x$ for $x \in A(\pi)$ and $b \in \aB$. Since $\gamma_\pi$ maps into the center of $\aB \tensor \End(V)$, we obtain for $\pi$ odd that $b\,.\,x = ((b \circ h)\tensor\one) x = x \,.\, (b \circ h)$ and for $\pi$ even likewise $b\,.\,x = x\,.\,b$. Hence, for every $x \in \aA_0 = C(\SU_2)$ and $b \in \aB = C(X)$ we have
	\begin{align*}
		bx &= xb
		&
		&\text{if $x$ is even,}
		\\
		bx &= x(b \circ h) 
		&
		&\text{if $x$ is odd.}
	\end{align*}
	Theses relations determine the C\Star algebra $\aA$ and the action of $G$ on $\aA$ uniquely, that is, $\aA$ is the unique C\Star algebra generated by $C(\SU_2)$ and $C(X)$ subject to the above relations. The algebra $\aA$ is a particular case of the a twisted tensor product (see \cite{Kasp81,MeSuWo14}).
\end{expl}
\begin{rmk}
	It is a well-known fact from homotopy theory that there exists, up to isomorphy, exactly one principal $\SU_2$-bundle over the 3-sphere $\mathbb{S}^3$, namely the trivial principal $\SU_2$-bundle $\mathbb{S}^3\times \SU_2$. It is used as a toy model for Chern-Simons theory developed by Witten \cite{Wi89} to derive a 3-dimensional quantum field theory in order to give an intrinsic definition of the Jones polynomial and its generalizations dealing with knots in three dimensional space. Example~\ref{expl:SU_2} shows that the noncommutative setting provides more possibilities of constructing ``noncommutative principal $\SU_2$-bundles'' over~$\mathbb{S}^3$.
\end{rmk}

\pagebreak[3]
\section{Construction of Cleft Systems}
\label{sec:construction}

In the previous section we have seen that factor systems provide an invariant of weakly cleft C\Star dynamical systems. In the following we will see that they actually provide a full classification. That is, we will show that for every factor system $(\gamma, \omega)$ for $(\aB,G)$ there actually is a weakly cleft C\Star dynamical system $(\aA, G, \alpha)$ with factor system $(\gamma, \omega)$. We would like to mention that the construction principle is not limited to weakly cleft actions and can be carried out more abstractly for general actions of compact quantum groups on C\Star algebras (see \cite{Ne13}). However, our restricted setting allows some simplifications.

We will split the construction into three steps. In the first step we will show how a factor system gives rise to a multiplication and hence an algebra. The second step will concern the construction of an involution by exploiting the Hilbert module structure. In the last step we proceed along the lines sketched in Section~\ref{sec:decomp} to finally construct the C\Star algebra and the dynamical system.

\subsection{Associativity and Factor Systems}
\label{sec:associativity}

Once and for all let us fix a compact group $G$ and a unital C\Star algebra $\aB$, and let $(\gamma, \omega)$ be a pair consisting of a family $\gamma = (\gamma_\pi)_{\pi \in \hat G}$ of unital \Star homomorphisms $\gamma_\pi:\aB \to \aB \tensor \End(V_\pi)$ and a family of isometries $\omega(\pi,\rho) \in \aB \tensor \End( V_\pi \tensor V_\rho)$ for all $\pi,\rho \in \hat G$. For the moment we do not assume that $(\gamma, \omega)$ is a factor system. We extend both families naturally to arbitrary \mbox{representations} of $G$, that is, we decompose given representations $(\pi, V)$, $(\rho, W)$ into direct sums of irreducible components and define the corresponding maps $\gamma_\pi$ and elements $\omega(\pi, \rho)$ componentwise. Equivalently, $\gamma_\pi:\aB \to \aB \tensor \End(V)$ and $\omega(\pi,\rho) \in \aB \tensor \End(V \tensor W)$ are the unique \Star homomorphisms and unitary elements, respectively, such that
\begin{align}
	\label{eq:intertwining_rel}
	\gamma_\pi(b) v &= v \gamma_\sigma(b)
	&
	&\text{and}
	&
	\omega(\pi,\rho) (v \tensor w) &= (v \tensor w) \omega(\sigma, \sigma')
\end{align}
for all isometric intertwiners $v:V_\sigma \to V$ and $w:V_{\sigma'} \to W$ with irreducible representations $\sigma, \sigma' \in \hat G$ and all $b \in \aB$.

For each representation $(\pi,V)$ of $G$ we consider the correspondence $\aB \tensor \End(V)$ over $\aB$ given by the usual right multiplication, the $\aB$-valued inner product $\scal{x,y}_\aB := \tfrac{1}{d_\pi} (\id \tensor \Tr)(x^* y)$, and the left multiplication 
\begin{equation*}
	b \,.\, x := \gamma_\pi(b)x 
\end{equation*}
for all $b \in \aB$ and $x,y \in \aB \tensor \End(V)$. We write $\norm{x}_2 := \scal{x,x}_\aB^{1/2}$ for the corresponding norm of $x \in \aB \tensor \End(V)$ and for an irreducible representation $\sigma \in \hat G$, we write $P_{\sigma, \pi}$ for the map $P_{\sigma, \pi}:~\aB \tensor \End(V) \to \aB \tensor \End(V_\sigma)$ given by $P_{\sigma, \pi}(x) := \sum_{k=1}^m v_k^* x v_k^{}$, where $v_1, \dots, v_m:V_\sigma \to V$ is an orthonormal basis of intertwiners. The map does not depend on the choice of the orthonormal basis. Moreover, the map $P_{\sigma, \pi}$ is adjointable with adjoint given by \mbox{$P_{\sigma, \pi}^\star(z) = \frac{d_\pi}{d_\sigma} \sum_{k=1}^m v_k^{} y v_k^*$} for all $z \in \aB \tensor \End(V_\sigma)$.

For each pair of representations $(\pi,V)$ and $(\rho, W)$ of $G$ we define a linear map by
\begin{gather*}
	m_{\pi, \rho}: [\aB \tensor \End(V)] \tensor_\aB [\aB \tensor \End(W)] \longrightarrow \aB \tensor \End(V \tensor W)
	\\
	m_{\pi,\rho}(x \tensor y) := \omega(\pi,\rho) \cdot (\id_\pi \tensor \gamma_\rho)(x) \cdot (\one_\pi \tensor y).
\end{gather*}
We refer to $m_{\pi, \rho}$ as \emph{multiplication} map. 
For sake of a concise notation we will frequently drop the subindex and simply write $m$ if the respective domains are not substantial and clearly determined by the context. Moreover, we amplify in the canonical way. With this simplifications the map $m_{\pi, \rho}$ takes the form
\begin{equation*}
	m_{\pi, \rho}(x \tensor y) = \omega(\pi, \rho) \,  \gamma_\rho(x)  y
\end{equation*}
for all $x \in \aB \tensor \End(V)$ and $y \in \aB \tensor \End(W)$. 

\pagebreak[4]
\begin{lemma}
	The map $m_{\pi, \rho}$ is an adjointable isometry. 
	\label{lem:m_isometry}
\end{lemma}
\begin{proof}
	Recall that for any C\Star algebra $\aA$, equipped with the natural right inner product $\scal{x,y}_\aA := x^*y$ ($x,y \in \aA$), and for any isometry $s \in \aA$ the map $\aA \to \aA$, $a \mapsto sa$ is an adjointable isometry.  
	In particular, it follows that the map $z \mapsto \omega(\pi,\rho) \,z$ on $\aB \tensor \End(V \tensor W)$ is an adjointable isometry. Therefore, the assertion follows from the fact that the tensor product $[\aB \tensor \End(V)] \tensor_\aB [\aB \tensor \End(W)]$ is canonically isometrically isomorphic to $\aB \tensor \End(V \tensor W)$ via the isomorphism $x \tensor y \mapsto \gamma_\rho(x)y$ for all \mbox{$x \in \aB \tensor \End(V)$} and $y \in \aB \tensor \End(W)$. 
\end{proof}

\pagebreak[3]
\begin{lemma}
	\label{lem:assoc+}
	Suppose $(\gamma,\omega)$ is a factor system. Then, for all representations $\pi, \rho$ of~$G$ the map $m_{\pi, \rho}$ is a $\aB$-bimodule map.
	Moreover, we have 
	\begin{equation}
		\label{eq:assoc+}
		(\id_\pi \tensor m_{\rho, \sigma}) \circ (m_{\pi, \rho}^\star \tensor \id_\sigma)
		=
		m_{\pi, \rho \tensor \sigma}^\star \circ m_{\pi \tensor \rho, \sigma}.
	\end{equation}
	for all representations $\pi, \rho, \sigma$ of $G$.
	In particular, the family of maps $m_{\pi, \rho}$ is associative, \ie, 
	\begin{equation*}
		m_{\pi, \rho \tensor \sigma} \circ (\id_\pi \tensor m_{\rho, \sigma}) 
		=
		m_{\pi \tensor \rho, \sigma} \circ (m_{\pi, \rho} \tensor \id_\sigma).
	\end{equation*}
\end{lemma}
\begin{proof}
	The right module property of $m_{\pi, \rho}$ is obviously satisfied even without assuming that $(\gamma,\omega)$ is a factor system.  Furthermore, by the coaction condition of the factor system, Equation~\eqref{eq:gamma}, we obtain 
	\begin{align*}
		m \bigl( b \,.\, (x \tensor y) \bigr) 
		&= m \bigl( \gamma_\pi(b)x \tensor y \bigr)
		= \omega(\pi,\rho)  \, \gamma_\rho \bigl( \gamma_\pi(b) x \bigr) \, y 
		\\
		&\overset{\eqref{eq:gamma}}= 
		\gamma_{\pi \tensor \rho}(b) \, \omega(\pi, \rho) \,  \gamma_\rho(x) y
		= b \,.\, m(x \tensor y)
	\end{align*}
	for all $b \in \aB$, $x \in \aB \tensor \End(V)$, and $y \in \aB \tensor \End(W)$. This shows that $m_{\pi, \rho}$ is indeed a $\aB$\ndash bimodule map. In particular, the left and right hand side of \eqref{eq:assoc+} are well-defined maps on the correspondence $[\aB \tensor \End(V_\pi \tensor V_\rho)] \tensor_\aB [\aB \tensor \End(V_\sigma)]$. 
	In order to verify \eqref{eq:assoc+} for representations $\pi, \rho, \sigma$ of $G$, we canonically identify the tensor product of the three correspondences $\aB \tensor \End(V_\pi)$, $\aB \tensor \End(V_\rho)$, and $\aB \tensor \End(V_\sigma)$ over $\aB$ with the correspondence $\aB \tensor \End(V_\pi \tensor V_\rho \tensor V_\sigma)$ via the isometric isomorphism $x \tensor y \tensor z \mapsto \gamma_\sigma \bigl( \gamma_\rho(x) y \bigr) z$. Then  the cocycle condition, Equation~\eqref{eq:u}, implies  that
	\begin{align*}
		\MoveEqLeft
		(\id_\pi \tensor m_{\rho, \sigma})\bigl( (m_{\pi, \rho}^\star \tensor \id_\sigma)(x) \bigr) 
		= 
		\omega(\rho,\sigma) \, \gamma_\sigma \bigl( \omega(\pi, \rho) \bigr)^* \, x,
		\\
		&\overset{\eqref{eq:u}}=
		\omega(\pi, \rho \tensor \sigma)^* \, \omega(\pi \tensor \rho, \sigma) \, x
		=
		m_{\pi, \rho \tensor \sigma}^\star \bigl( m_{\pi \tensor \rho, \sigma}(x) \bigr)
	\end{align*}
	for all $x \in \aB \tensor \End(V_\pi \tensor V_\rho \tensor V_\sigma)$, which proves Equation~\eqref{eq:assoc+}. Multiplying \eqref{eq:assoc+} with $(m_{\pi, \rho} \tensor \id_\sigma)$ from the right and $m_{\pi, \rho \tensor \sigma}$ from the left yields 
	\begin{equation*}
		m_{\pi, \rho \tensor \sigma} (\id_\pi \tensor m_{\rho, \sigma}) 
		=
		m_{\pi, \rho \tensor \sigma} \circ m_{\pi, \rho \tensor \sigma}^\star \circ m_{\pi \tensor \rho, \sigma} \circ (m_{\pi, \rho} \tensor \id_\sigma).
	\end{equation*}
	Since the left hand side of this equation and the term $S := m_{\pi \tensor \rho, \sigma} \circ (m_{\pi, \rho} \tensor \id_\sigma)$ on the right hand side are both isometries, the orthogonal projection $m_{\pi, \rho \tensor \sigma} \circ m_{\pi, \rho \tensor \sigma}^\star$ acts trivially on the rage of $S$. Therefore, we may cancel $m_{\pi, \rho \tensor \sigma} \circ m_{\pi, \rho \tensor \sigma}^\star$, which shows associativity.
\end{proof}

\begin{rmk}
	The normalization condition $\omega(1,1) = \one_\aB$ of Definition~\ref{defn:factor_sys} is equivalent to the fact that for the trivial representation the multiplication map~$m$ recovers the $\aB$-bimodule structure, that is, for a representation $(\pi,V)$ of $G$ we have
	\begin{align*}
		m_{1,\pi}(b \tensor x) &= b \,.\, x = \gamma_\pi(b)x,
		&
		m_{\pi, 1}(x \tensor b) &= x \,.\, b = x(b \tensor \one_\pi)
	\end{align*}
	for all elements $x \in \aB \tensor \End(V)$ and $b \in \aB$. In particular, in this case $m_{1,1}$ coincides with the usual multiplication of~$\aB$.
	\label{rmk:subalg_B}
\end{rmk}

In order to define an algebra, we consider the algebraic direct sum of the correspondences $\aB \tensor \End(V_\pi)$ for $\pi \in \hat G$:
\begin{equation*}
	A := \bigoplus_{\pi \in \hat G} \aB \tensor \End(V_\pi).
\end{equation*}
The left and right action of $\aB$ are given componentwise and the $\aB$-valued inner product is $\scal{x,y}_\aB = \sum_{\sigma \in \hat G} \scal{x_\sigma, y_\sigma}_\aB$, where $x_\sigma$ and $y_\sigma$ denote the components of $x$ and $y$, respectively. 
On $A$ we define the product of $x \in \aB \tensor \End(V_\pi)$ and $y \in \aB \tensor \End(V_\rho)$ with $\pi, \rho \in \hat G$ by 
\begin{equation*}
	x \bullet y := \sum_{\sigma \in \hat G} P_{\sigma, \pi \tensor \rho} \bigl( m(x \tensor y) \bigr).
\end{equation*}
Bilinear extension then yields a bilinear map $(x,y) \mapsto x \bullet y$ on~$A$. For $x \in \aB$ or $y \in \aB$ this product coincides with the left or right action of $\aB$, respectively.

Due to the intertwining relations \eqref{eq:intertwining_rel} the family of maps $m_{\pi, \rho}$ behaves nicely with respect to intertwiners. It is straightforward to check that for all representations $\pi, \rho$ of $G$ and every intertwiner $v:V_\sigma \to V_\pi$ we have
\begin{equation}
	\label{eq:intertw_rel_m}
	\begin{aligned}
	m(v x v^* \tensor y) 
	&= (v \tensor \one_\rho) \, m(x \tensor y) \, (v \tensor \one_\rho)^*,
	\\
	m(v^* z v \tensor y) 
	&= (v \tensor \one_\rho)^* m(z \tensor y) (v \tensor \one_\rho)
	\end{aligned}
\end{equation}
for all $x \in \aB \tensor \End(V_\pi)$ and $y \in \aB \tensor \End(V_\rho)$, and similar equations hold for adjoining intertwiners in the right tensor factor. Using this relations, the associativity of Lemma~\ref{lem:assoc+} and some algebra straightforwardly yields the following:

\begin{lemma}
	Suppose $(\gamma, \omega)$ is a factor system. Then the product $\bullet$ on $A$ is associative, that is, $(A, \bullet)$ is an algebra.
\end{lemma}

\subsection{Constructing the Involution}

Throughout the remainder of the section let us assume that $(\gamma,\omega)$ is a factor system.  In order to define an involution on the algebra $A$, let consider a fixed $\pi \in \hat G$. We put $p_\pi := P_{1,\pi \tensor \bar\pi}^*(\one_\aB)$, that is, $p_\pi$ is the $d_\pi^2$-th multiple of the orthogonal projection onto the fixed point space of the representation $\pi \tensor \bar\pi$. For every \mbox{$x \in \aB \tensor \End(V_\pi)$} we define an element in $\aB \tensor \End(\bar V_\pi)$ by
\begin{equation*}
	J(x) := (L_x^\star \circ m_{\pi, \bar\pi}^\star)(p_\pi), 
\end{equation*}
where we denote by $L_x:\aB \tensor \End(\bar V_\pi) \to [\aB \tensor \End(V_\pi)] \tensor_\aB [\aB \tensor \End(\bar V_\pi)]$ the adjointable map given by $L_xy := x \tensor y$. Extending this antilinearly to all summands provides an antilinear map \mbox{$J:A \to A$}. On the subalgebra $\aB \subseteq A$ (\ie, for $\pi =1$) we immediately find $J(b) = b^*$ for all $b \in \aB$. For the other summands, however, $J$~does not coincide with the usual involution on $\aB \tensor \End(V_\pi)$ in general.

\begin{thm}
	\label{thm:inner_product}
	For all $x,y,z \in A$ we have
	\begin{equation*}
		\scal{J(x) \bullet y, \, z}_\aB = \scal{y, \, x \bullet z}_\aB.
	\end{equation*}
\end{thm}
\begin{proof}
	Let $x \in \aB \tensor \End(V_\pi)$, $y \in \aB \tensor \End(V_\rho)$, and \mbox{$z \in \aB \tensor \End(V_\sigma)$} with $\pi, \rho, \sigma \in \hat G$. Since $(\gamma,\omega)$ is a factor system, the multiplication maps satisfy \eqref{eq:assoc+} of Lemma~\ref{lem:assoc+}. Then, writing $\tilde z := P_{\sigma, \bar\pi \tensor \rho}^\star(z)$ for short, we obtain
	\begin{align*}
		\scal{J(x) \bullet y, z}_\aB
		&=
		\scal{ m_{\bar\pi, \rho} \bigl( J(x) \tensor y \bigr), \; \tilde z}_\aB
		= 
		\scal{m_{\bar\pi, \rho} \bigl( L_x^\star m_{\pi, \bar\pi}^\star(p_\pi) \tensor y \bigr), \; \tilde z}_\aB
		\\
		&=
		\scal{ L_x^\star (\id_\pi \tensor m_{\bar\pi, \rho}) (m_{\pi, \bar\pi}^\star \tensor \id_\rho)(p_\pi \tensor y), \; \tilde z}_\aB
		\\
		&\overset{\eqref{eq:assoc+}}=
		\scal{L_x^\star m_{\pi, \bar\pi \tensor \rho}^\star m_{\pi \tensor \bar\pi, \rho}(p_\pi \tensor y), \; \tilde z}_\aB
		\\
		&=
		\scal{m_{\pi \tensor \bar\pi, \rho}(p_\pi \tensor y), \; m_{\pi, \bar\pi\tensor \rho}(x \tensor \tilde z)}_\aB		.
		\\
		&\overset{\eqref{eq:intertwining_rel}}= 
		\scal{(P_{\id, \pi \tensor \bar\pi}^\star \tensor \id_\rho)(y), \; (\id_\pi \tensor P_{\sigma, \bar\pi \tensor \rho}^\star) m(x \tensor z)}_\aB.
	\end{align*}
	We fix an isometric intertwiner $w:\C \to V_{\pi \tensor \bar\pi}$ and we choose an arbitrary orthonormal bases of intertwiners $v_1, \dots, v_m:V_\sigma \to \bar V_\pi \tensor V_\rho$. Then $u_k := \sqrt{\frac{d_\pi d_\rho}{d_\sigma}}(\one_\pi \tensor v_k)^*(w \tensor \one_\rho)$ for $1 \le k \le m$ form an orthonormal basis of intertwiners from $\rho$ to $\pi \tensor \sigma$. 
	Orthonormality follows from the fact that for every $T \in  \End(\bar V_\pi)$ we have $w^*(\one_\pi \tensor T) w = \tfrac{1}{d_\pi} \Tr(T)$. Completeness holds because the dimension of the intertwiner space from $\rho$ to $\pi \tensor \sigma$ coincides with the dimension of the intertwiner space from $\sigma$ to $\rho \tensor \bar \pi$.
	With this intertwiners we find $P_{\rho, \pi \tensor \sigma} = (P_{\id, \pi \tensor \bar\pi} \tensor \id_\rho) (\id_\pi \tensor P_{\sigma, \bar\pi \tensor \rho}^\star)$. Continuing the above computation, we obtain 
	\begin{equation*}
		\scal{J(x) \bullet y, z}_\aB = \scal{y, P_{\rho, \pi \tensor \sigma} m(x \tensor z)}_\aB = \scal{y, x \bullet z}_\aB.
	\end{equation*}
	The assertion for arbitrary $x,y,z \in A$ then follows by extending (anti-)linearly.
\end{proof}

\begin{cor}
	\label{cor:P0_product}
	Let $P_0:A \to \aB$ denote the orthogonal projection onto the direct summand $\aB \subseteq A$ corresponding to the trivial representation. Then for all $x,y \in A$ we have
	\begin{equation*}
		\lscal\aB{x,y} = P_0 \bigl( J(x) \bullet y \bigr).
	\end{equation*}
\end{cor}
\begin{proof}
	The element $\one_\aB$ is a unit for the multiplication of $A$ and fixed by the orthogonal projection $P_0$. Hence by Theorem~\ref{thm:inner_product} we obtain
	\begin{equation*}
		\scal{x,y}_\aB
		= \scal{\one_\aB, J(x) \bullet y}_\aB
		= \scal*{\one_\aB, P_0 \bigl( J(x) \bullet y \bigr)}_\aB
		= P_0 \bigl( J(x) \bullet y \bigr).
		\qedhere
	\end{equation*}
\end{proof}
\begin{rmk}
	\label{rmk:unique_J}
	We would like to mention that by Corollary~\ref{cor:P0_product} the element $J(x)$ can be equivalently be characterized as the unique element of $A$ satisfying $\scal{J(x),y}_\aB = P_0(x \bullet y)$ for all $y \in A$.
\end{rmk}

\begin{cor}
	\label{cor:involutive_alg}
	$A$ is an involutive algebra, \ie, for all $x,y \in A$ we have
	\begin{align*}
		J\bigl( J(x) \bigr) &= x,
		&
		J(x \bullet y) &= J(y) \bullet J(x).
	\end{align*}
\end{cor}
\begin{proof}
	By applying Theorem~\ref{thm:inner_product} twice we get
	\begin{gather*}
		\lscal\aB{J \bigl( J(x) \bigr), \; z}
		= \scal{\one_\aB, J(x) \bullet z}_\aB 
		= \scal{x,z}_\aB
		\shortintertext{and}
		\lscal\aB{J(x \bullet y), \; z}
		= \lscal\aB{\one_\aB, \; x \bullet y \bullet z}
		= \lscal\aB{J(x), \; y \bullet z}
		= \lscal\aB{J(y) \bullet J(x), \; z}
	\end{gather*}
	for all $z \in A$. Since the inner product separates points, this yields the assertion.
\end{proof}

\subsection{Construction of Free Actions}
\label{constr free action}

Having the algebra $A$ and the involution on $A$ in hands, the construction of the C\Star al\-ge\-bra and the cleft action follows the outline presented in Section~\ref{sec:decomp}. We consider $A$ as a right pre-Hilbert $\aB$-module. The inner product is positive definite but, unless $G$ is finite, $A$ is not closed with respect to the induced norm
\begin{equation*}
	\norm{x}_2 := \norm{\lscal\aB{x,x}}^{1/2}_{\mathrm{op}} = \norm{P_0 \bigl( J(x) \bullet x \bigr)}_{\mathrm{op}}^{1/2},
	\qquad \qquad x \in A.
\end{equation*}
We denote by $\bar A$ the completion with respect to this norm. Equivalently, the correspondence $\bar A$ over $\aB$ is the direct sum of the previously discussed correspondences $\aB \tensor \End(V_\pi)$ with \mbox{$\pi \in \hat G$}.

Theorem~\ref{thm:inner_product} allows us to extend each left multiplication by $x \in A$ to an adjointable map on the completion $\bar A$, which we again denote by $\lambda[x]:\bar A \to \bar A$. Then the map 
\begin{equation*}
	\lambda: A \to \End(\bar A), 
	\quad 
	x \mapsto \lambda[x]
\end{equation*}
is an representation of the \Star algebra $A$ by adjointable operators on~$\bar A$. The vector $\one_\aB \in \bar A$ is clearly cyclic and separating for the subalgebra $\lambda(A)$. In particular, $\lambda$ is faithful and hence isometric on $\aB \subseteq A$.
We denote by $\aA$ the C\Star algebra generated by the range of~$\lambda$. 
To simplify the notation we identify the algebra $(A, \bullet)$ with the subalgebra $\lambda(A) \subseteq \aA$.  For sake of clarity, we extend the notation for the \Star algebra $A$ to the C\Star algebra $\End(\bar A)$, that is, for elements $x,y \in \End(\bar A)$ we write \mbox{$x \bullet y$} for their product and for $x \in \aA$ we write $J(x)$ for its adjoint.  Since for $x \in A$ we have $\norm{\lambda[x]}^2 \ge \norm{\lambda[x]\one_\aB}_2^2 = \norm{x}_2^2$, we may regard the C\Star algebra $\aA$ as a subset of~$\bar A$, so that $A \subseteq \aA \subseteq \bar A$.

On each direct summand $\aB \tensor \End(V_\rho) \subseteq \bar A$, $\pi \in \hat G$, we have a continuous action of $G$ by the usual right multiplication
\begin{equation*}
	U_g(x) := x \, (\one_\aB \tensor \pi_g^*)
\end{equation*}
for $g \in G$ and $x \in \aB \tensor \End(V_\pi)$. For each $g\in G$ the map $U_g$ is unitary on $\aB \tensor \End(V_\rho)$ with respect to the $\aB$-valued inner product. Taking direct sums and continuous extension, we obtain a strongly continuous unitary representation $g \mapsto U_g \in \End(\bar A)$ of the group $G$ on the correspondence $\bar A$. We denote by $\alpha = (\alpha_g)_{g\in G}$ the associated automorphism group on $\End(\bar A)$, \ie, we put
\begin{equation*}
	\alpha_g(x) := U_g \bullet x \bullet U_g^\star, 
	\qquad \qquad
	x \in \End(\bar A).
\end{equation*}
For each element $x \in A \subseteq \aA$ we have $\alpha_g(x) = U_g(x)$, that is, on the algebra $\aA \subseteq \bar A$ the actions $\alpha = (\alpha_g)_{g \in G}$ and $U = (U_g)_{g \in G}$ coincide. Since $A$ is dense in~$\aA$, it follows that $g \mapsto \alpha_g(x)$ is continuous for every $x \in \aA$. Summarizing, we have constructed a C\Star dynamical system $(\aA, G, \alpha)$. 

\begin{thm}
	The C\Star dynamical system $(\aA, G, \alpha)$ is weakly cleft with fixed point algebra $\aB$ and factor system $(\gamma, \omega)$.
	\label{thm:factor_2_sys}
\end{thm}
\begin{proof}
	In the first part of the proof we will show that for the action $U = (U_g)_{g \in G}$ on $\bar A$ the isotypic component of $\pi \in \hat G$ is given by the direct summand $\aB \tensor \End(\bar V_\pi)$. This verifies in particular that $\aB$ is the fixed point space. By Theorem~\ref{thm:inner_product} it follows that the canonical $\aB$-valued inner product on $\aA$ coincides with the inner product on the larger space $\bar A$, that is,
	\begin{equation*}
		\int_G \alpha_g \bigl(J(x) \bullet y \bigr) \;dg 
		= P_0 \bigl( J(x) \bullet y) 
		\overset{\ref{thm:inner_product}}= \scal{x,y}_\aB.
	\end{equation*}
	for all $x,y  \in \aA$.
	Using Lemma~\ref{lem:cleft_isom} we conclude that the dynamical system $(\aA, G, \alpha)$ is weakly cleft. In the second part of the proof we confirm that $(\gamma, \omega)$ is indeed a factor system of $(\aA, G, \alpha)$. 
	\begin{enumerate}
	\item 
		We consider the action $U = (U_g)_{g \in G}$ on the correspondence $\bar A$ over $\aB$. Obviously, all elements of $\aB \subseteq \bar A$ are fixed by the action. For an element $x \in \aB \tensor \End(V_\pi)$ in a direct summand of $\bar A$ with $\id \neq \pi \in \hat G$ the integral $\int_G U_g(x) \; dg = \int_G x \pi_g^* \; dg$ vanishes because $\pi$ is irreducible. Taking linear combinations and continuous extensions then shows that $\int_G U_g(x) \; dg = P_0(x)$ for all $x \in \bar A$. Adapting the arguments, we obtain for every $\pi \in \hat G$ and $x \in \bar A$
		\begin{equation*}
			d_\pi \int_G \Tr(\pi_g^*) U_g(x) \; dg = P_{\bar\pi}(x),
		\end{equation*}
		where $P_{\bar\pi}:\bar A \to \aB \tensor \End(\bar V_\pi)$ denotes the orthogonal projection onto the direct summand $\aB \tensor \End(\bar V_\pi) \subseteq \bar A$. We conclude that $\aB \tensor \End(\bar V_\pi)$ is the isotypic component of~$\pi$ in~$\bar A$ and hence also in $\aA \subseteq \bar A$.
	\item
		It remains to verify that $(\gamma, \omega)$ is a factor system of $(\aA, G, \alpha)$. According to the first part of the proof the isotypic component of $\bar\pi \in \hat G$ in $\aA$ is given by $\aB \tensor \End(V_\pi)$. As a right $\aB$-module $\aB \tensor \End(V_\pi)$ is generated by the element $\one_\aB \tensor \one_\pi$ of norm~1. The natural isomorphism of the isotypic and the generalized isotypic component together with Lemma~\ref{lem:cleft_isom} therefore provides us with the non-degenerate isometry $s_\pi \in A_2(\pi) \subseteq \aA \tensor \End(V_\pi)$ given by
		\begin{equation*}
			s_\pi 
			= d_\pi \int_G \one_\aB \tensor \pi_g^* \tensor \pi_g \; dg
			= (\one_\aB \tensor F),
		\end{equation*}
		where $F \in \End(V) \tensor \End(V)$ denotes the tensor flip. This family $s = (s_\pi)_{\pi \in \hat G}$ of isometries then gives rise to a factor system $(\tilde \gamma, \tilde \omega)$ by Equations~\eqref{eq:gamma_concrete} and~\eqref{eq:u_concrete}. For each $\pi \in \hat G$ the homomorphism $\tilde\gamma_\pi$ is uniquely determined by the left action of $\aB$ on $s_\pi$. For convenience we also write $\bullet$ for the multiplication on all $\aA \tensor \End(V_\pi)$, $\pi \in \hat G$. Then the left action of $\aB$ on $s_\pi$ is given by
		\begin{align*}
			(b \tensor \one_\pi) \bullet s_\pi 
			&= d_\pi \int_G  \bigl( b \bullet (\one_\aB \tensor \pi_g^*) \bigr) \tensor \pi_g 
			= d_\pi \int_G \bigl( \gamma_\pi(b) (\one_\aB \tensor \pi_g^*) \bigr) \tensor \pi_g
			\\
			&= \bigl( \gamma_\pi(b) \tensor \one_\pi \bigr) (\one_\aB \tensor F)
			= s_\pi (\one_\aB \tensor F) \bigl( \gamma_\pi(b) \tensor \one_\pi \bigr) (\one_\aB \tensor F)
			\\
			&= s_\pi \bullet \gamma_\pi(b)
		\end{align*}
		for every $b \in \aB$. Consequently, $\tilde \gamma_\pi = \gamma_\pi$. With some algebra involving the multiplication $\bullet$, which we leave to the reader, it is then straightforward to show that we also have
		\begin{equation*}
			s_\pi \bullet s_\rho
			= s_{\pi \tensor \rho} \bullet \omega(\pi,\rho)
		\end{equation*}
		for all $\pi, \rho \in \hat G$. This proves $\tilde \omega(\pi,\rho) = s_{\pi \tensor \rho}^* \bullet s_\pi \bullet s_\rho = \omega(\pi,\rho)$. Summarizing, we have shown that $(\gamma, \omega)$ is indeed the factor system associated with the isometries $s_\pi$, $\pi \in \hat G$, which finishes the proof.
		\qedhere
	\end{enumerate}
\end{proof}

This concrete representation theorem of weakly cleft actions finally allows us as a corollary to complete the proof of Corollary~\ref{cor:ergodic_subsyst}, which we asserted in Section~\ref{sec:factor_sys}. For convenience we repeat the relevant statement as a reminder in a slightly more general form.
\begin{cor}
	Let $(\aA, G, \alpha)$ be a C\Star dynamical system and $\aB_0 \subseteq \aB$ a unital C\Star sub\-al\-ge\-bra of fixed points. Furthermore, let $(\gamma, \omega)$ be a factor system such that 
	\begin{align*}
		\gamma_\pi(\aB_0) &\subseteq \aB_0 \tensor \End(V_\pi),
		&
		\omega(\pi,\rho)  &\in \aB_0 \tensor \End(V_\pi \tensor V_\rho),
	\end{align*}
	for all $\pi, \rho\in \hat G$. Then there is an $\alpha$-invariant subalgebra $\aA_0 \subseteq \aA$ such that the restricted dynamical  system $(\aA_0, G, \alpha|_{\aA_0})$ is weakly cleft and $\aA$ is generated by $\aA_0$ and $\aB$.
	\label{cor:remaining}
\end{cor}
\begin{proof}
	By Theorem~\ref{thm:factor_2_sys} and Theorem~\ref{thm:factor_sys_concrete} we may assume that $\aA$ admits an $\alpha$\ndash invariant dense subalgebra of the form $A = \bigoplus_{\pi \in \hat G} \aB \tensor \End(V_\pi)$ with the multiplication
	\begin{equation*}
		x \bullet y = \sum_{\sigma \in \hat G} P_\sigma \bigl( \omega(\pi, \rho) \,  \gamma_\rho(x) \, y \bigr)
	\end{equation*}
	and the action $\alpha_g(x) = x (\one_\aB \tensor \pi_g^*)$ for all $x \in \aB \tensor \End(V_\pi)$, $y \in \aB \tensor \End(V_\rho)$, $g \in G$, and $\pi, \rho \in \hat G$. By the hypotheses on $(\gamma, \omega)$ the subset
	\begin{equation*}
		A_0 := \bigoplus_{\pi \in \hat G} \aB_0 \tensor \End(V_\pi) 
	\end{equation*}
	is an $\alpha$-invariant \Star subalgebra of~$A$ (\cf also Remark~\ref{rmk:unique_J}). Let $\aA_0 \subseteq \aA$ be the closure of $A_0$ with respect to the operator norm. By Theorem~\ref{thm:factor_2_sys}, the dynamical system restricted to $\aA_0$ is weakly cleft with the restriction of $(\gamma, \omega)$ to $\aB_0$ as factor system. The space $A$ is linearly generated by all products $x \bullet b = (b \tensor x)$ with $x \in \one_\aB \tensor \End(V_\pi)$, $\pi \in \hat G$, and $b \in \aB$. Since $A_0$ contains all elements $x \in \one_{\aB} \tensor \End(V_\pi)$, $\pi \in \hat G$, the algebra $A$ is generated by $A_0$ and $\aB$. Taking closures, $\aA$ is generated by $\aA_0$ and~$\aB$.
\end{proof}

\section*{Acknowledgment}

We would like to acknowledge the Center of Excellence in Analysis and Dynamics
Research (Academy of Finland, decision no.~271983 and no.~1138810) for supporting this research.


\bibliographystyle{abbrv}
\bibliography{short,part2}

\end{document}